\documentclass[11pt,twoside,reqno]{amsart}

\usepackage{amsmath, amssymb, amsthm, enumitem, esint}
\usepackage[hyperindex]{hyperref}
\usepackage{tensor}

\usepackage{hyperref}
\hypersetup{bookmarksdepth=3}

\newcounter{mtheorem}
\newtheorem{mtheorem}[mtheorem]{Theorem}

\newcommand{\noopsort}[1]{}

\newcommand{\EMPTY}[1]{}

\newtheorem{Theorem}[equation]{Theorem}
\newtheorem{Lemma}[equation]{Lemma}
\newtheorem{Corollary}[equation]{Corollary}
\newtheorem{Proposition}[equation]{Proposition}

\newtheorem{Question}[equation]{Question}
\newtheorem{Conjecture}[equation]{Conjecture}

\theoremstyle{definition}
\newtheorem{Definition}[equation]{Definition}
\theoremstyle{remark}

\numberwithin{equation}{section}

\begin{document}

\title{The Sign of Scalar Curvature on K\"ahler Blowups}
\date{\today}

\author{Garrett M. Brown}
\address{Department of Mathematics, University of California Berkeley, CA 94720, USA}
\email{gmbrown@berkeley.edu}

\date{\today}

\begin{abstract}
We show that if $(M,\omega)$ is any compact K\"ahler manifold, then the blowup of $M$ at any point furnishes a K\"ahler metric with scalar curvature globally and arbitrarily $C^0$-close to the scalar curvature of $\omega$. It follows that if $M$ admits a positive scalar curvature K\"ahler metric, then so do all of its blowups. This special case extends a result of N. Hitchin \cite{Hit} to surfaces and answers a conjecture of C. LeBrun \cite{LebSC} in the affirmative, consequently completing the classification of positive scalar curvature K\"ahler surfaces as being precisely those of negative Kodaira dimension (i.e. blowups of either the projective plane or a holomorphic bundle of projective lines over a Riemann surface).

\end{abstract}

\maketitle

\tableofcontents

\section{Introduction}

\subsection{Overview} If $M$ is a complex manifold and $p \in M$, let $Bl_pM$ denote the blowup of $M$ at $p$, and $\pi: Bl_pM \rightarrow M$ the blowdown map. Let $E$ denote the exceptional divisor, and $[E] \in H^2(M,\mathbb C)$ its Poincar\'e dual. The main result of this work is the following.

\begin{mtheorem}[Main Theorem]\label{mainthm}
    Let $(M,\omega)$ be a compact K\"ahler manifold of complex dimension $n \geq 2$ and scalar curvature $S(\omega)$. For any $p \in M$, there exists a sequence $\omega_i$ of K\"ahler metrics on $Bl_pM$ with
    \[|S(\omega_i) - S(\omega)| \rightarrow 0 \text{ uniformly on } Bl_pM \text{ as } i \rightarrow \infty,\]
    whose K\"ahler classes $[\omega_i] \in H^{1,1}(M)$ converge to $\pi^*[\omega]$ as $i \rightarrow \infty$ (we implicitly pull $S(\omega)$ back along $\pi$).
\end{mtheorem}

In particular, if $(M, \omega)$ has positive scalar curvature, then all its blowups also have K\"ahler metrics with positive scalar curvature (see Theorem ~\ref{psccor}). For K\"ahler surfaces, this finishes a conjecture of C. LeBrun \cite{LebSC}, completing a classification of compact positive scalar curvature K\"ahler surfaces and drawing an equivalence with positive scalar curvature Riemannian geometry. We now state it as a theorem.

\begin{mtheorem}\label{conj}
    Let $M$ be a compact complex surface of K\"ahler-type (see section 7). Then the following are equivalent.
    \begin{enumerate}
        \item $M$ has a K\"ahler metric for which the total scalar curvature is positive.
        \item $M$ has a K\"ahler metric of positive scalar curvature.
        \item $M$ has a Riemannian metric of positive scalar curvature.
        \item $M$ has Kodaira dimension $-\infty$.
        \item $M$ is obtained from $\mathbb C\mathbb P^2$ or $\mathbb P(E)$ by a finite sequence of blowups, where $E$ is a rank 2 holomorphic vector bundle over a compact Riemann surface.
    \end{enumerate}
\end{mtheorem}

N. Hitchin had already shown that the existence of a positive scalar curvature K\"ahler metric is preserved under blowup in complex dimension $n \geq 3$ \cite{Hit}. In Hitchin's work, via holomorphic normal coordinates, the metric on the blowup behaves somewhat like
\[\omega_{euc} + \epsilon p^*\omega_{FS},\]
where $p: Bl_0\mathbb C^n \rightarrow \mathbb C\mathbb P^{n-1}$ is the natural map, $\omega_{FS}$ the Fubini-Study metric, and $\epsilon > 0$ small. This metric on $Bl_0\mathbb C^n$ has positive scalar curvature for $n \geq 3$ which allows the argument to go through, but in $n = 2$ this metric is the scalar flat \emph{Burns metric}.

So, one needs a different approach in $n = 2$. To this end, our main inspiration for the proof is the many works on the problem of constructing extremal K\"ahler metrics on the blowup of a point across \cite{AP1} \cite{AP2} \cite{APS} \cite{Sz1} \cite{Sz2} \cite{DS}. Among these one can extract a few slightly different approaches for contructing K\"ahler metrics on a blowup with specified scalar curvature properties, the one we manage to apply to our situation is that of G. Sz\'ekelyhidi, especially as described in \cite{SzB}. The broad strategy is to first construct an initial family of metrics $\omega_\epsilon$ on the blowup, and then solve some appropriate equation for $\phi$ of the form
\begin{equation}\label{form}
S(\omega_\epsilon + \sqrt{-1}\partial\overline{\partial}\phi) - S(\omega) = \text{ something shrinking as } \epsilon \rightarrow 0.
\end{equation}
This naturally leads to studying estimates on the linearized scalar curvature, but for technical reasons we are pushed instead onto its adjoint. In the work on extremal metrics, one encounters obstructions, but the added flexibility in our problem allows us to avoid such issues. Our approach works in all dimensions, though there are some additional difficulties in complex dimension 2 as in the case of extremal metrics. In contrast with the work of Hitchin, in our approach the scalar flatness of the Burns metric is an asset rather than a detriment, as we crucially exploit the special properties of the linearized scalar curvature operator of a constant scalar curvature K\"ahler (cscK) metric.

\subsection{Organization}

We begin in section 2 by constructing the family $\omega_\epsilon$ on the blowup by gluing a scaled down asymptotically Euclidean, scalar flat metric on $Bl_0\mathbb C^n$ to the original metric $\omega$. In section 3 we introduce the weighted H\"older spaces on the blowup suitable to establish uniform estimates in the parameter $\epsilon$. Section 4 is devoted to proving uniform estimates on the inverse of some perturbation of the formal adjoint of linearized scalar curvature for the family $\omega_\epsilon$. To get better behavior in the gluing region when $n = 2$, in section 5 we perturb the family $\omega_\epsilon$ while preserving the uniform estimates proved in section 4. In section 6, we write down an equation of the form \eqref{form} which is solved via the contraction mapping principle, for which the estimates of section 4 are applied. Finally in section 7, among other things we discuss the resulting classification in $n = 2$ for positive scalar curvature K\"ahler metrics.

\subsection{Conventions} If $M$ is a complex manifold with complex structure $J$, and $g$ is a Riemannian metric, we define a real $(1,1)$-form
\[\omega(-,-) := g(J-,-),\]
and $g$ is said to be K\"ahler with respect to the complex structure if equivalently $d\omega = 0$ or $J$ is parallel with respect to $g$. In this case $\omega$ is called the K\"ahler form, and we usually just denote a K\"ahler manifold by $(M,\omega)$ since then $g$ is determined by $\omega$ and the complex structure. All tensors on $M$, including $g$ and $\omega$ will be extended to the complexified tangent bundle by extending symmetric bilinearly (some authors may instead take the Hermitian extension of $g$).

Some tensors on the complexified tangent bundle will vanish on some combination of entries in the $\pm\sqrt{-1}-$eigenspaces of the complex structure $J$. For example, the extension of $g$ will vanish when both entries are in the $\sqrt{-1}$, or $-\sqrt{-1}$ eigenspaces. Thus, when we have chosen a local holomorphic coordinate $z$, for convenience we may denote
\[g_{j\overline{k}} := g\left(\frac{\partial}{\partial z_j}, \frac{\partial}{\partial \overline{z}_k}\right).\]
Another example is the curvature tensor of a K\"ahler metric, which only has components $Rm_{i\overline{j}k\overline{\ell}}$. Except when indicated in such special cases, all traces will be taken over all components.

Due to the small parameter $\epsilon$ for the family $\omega_\epsilon$, we will often write estimates that involve many constants that change from line to line, but we will usually denote them as the same unless otherwise indicated, the important part is that they are independent of $\epsilon$ unless otherwise indicated.

\subsection{Acknowledgements}  First, the author would like to extend his sincere gratitude to his advisor, Song Sun, for thoughtfully suggesting this problem, many fruitful discussions throughout the course of the project, and constant support. He would also like to thank Gabor Sz\'ekelyhidi for explaining some points of his work on blowing up extremal metrics, Mingyang Li and Junsheng Zhang for helpful comments on initial versions of this paper, and Yifan Chen, Charles Cifarelli, Carlos Esparza, Yueqing Feng, Daigo Ito, Tobias Shin, and Keshu Zhou for stimulating discussions and encouragement. This work was supported by a National Science Foundation Graduate Research Fellowship, grant number DGE-2146752.

\section{The Metric $\omega_\epsilon$ on the Blowup}

In this section we construct a family of metrics on $Bl_pM$ with very nice analytic properties. This type of construction is done for example in \cite{Sz1}.

The space $Bl_pM$ is constructed by replacing a neighborhood of $p \in M$ with a neighborhood of the zero section of $O(-1) \cong Bl_0\mathbb C^n$. Thus, how we will construct $\omega_\epsilon$ is by gluing the existing metric on $M$ to a special metric on $Bl_0\mathbb C^n$.

On $Bl_0\mathbb C^n$ there exists a complete, asymptotically Euclidean, scalar flat K\"ahler metric $\eta$ known as the \textit{Burns-Simanca} metric \cite{LeMass}, \cite{Sim}. Using the usual holomorphic coordinates $Bl_0\mathbb C^n\setminus E \cong \mathbb C^n \setminus \{0\}$, when $n = 2$
\begin{equation}\label{omega2}
\eta = \sqrt{-1}\partial\overline{\partial}(|w|^2 + \log|w|).
\end{equation}
In the case of $n \geq 3$ the metric is less explicit; there is $\psi \in C^\infty(\mathbb C^n \setminus \{0\})$ such that $\psi = O(|w|^{4 - 2n})$ at infinity and
\begin{equation}\label{omega3up}
\eta = \sqrt{-1}\partial\overline{\partial}(|w|^2 + \psi(w)).
\end{equation}
Note that this presentation implies that in these coordinates the components of the Ricci tensor decay like $O(|w|^{-2n})$.

On the other hand we let $z$ be a holomorphic normal coordinate on $M$ centered at $p$, so that in this neighborhood
\begin{equation}\label{normal}
\omega = \sqrt{-1}\partial\overline{\partial}(|z|^2 + \varphi(z)),\> \varphi = O(|z|^4),
\end{equation}
and without loss of generality we assume the coordinate exists in the ball $B_2$, that is $|z| \leq 2$. We glue $Bl_0\mathbb C^n$ to $M_p := M \setminus \{p\}$ along the coordinate change $z = \epsilon w$, so $B_1 \subset M$ is identified with the ball of radius $\epsilon^{-1}$ in $Bl_0\mathbb C^n$. We will denote all balls around the exceptional divisor in $Bl_0\mathbb C^n$ with a tilde, so for example $B_1$ in $M$ is identified with $\tilde{B}_{\epsilon^{-1}}$ in $Bl_0\mathbb C^n$. Note that these are not metric balls rather just balls with respect to the coordinates $z$ and $w$.

Now let $\gamma$ be a smooth cutoff function which is $1$ outside of $B_2$ and $0$ inside of $B_1$ on $M$. Then we set $\gamma_1(z) = \gamma(z/r_\epsilon)$ and $\gamma_2 = 1 - \gamma_1$, where $r_\epsilon = \epsilon^\beta$ for some $\beta \in (1/2,2/3)$ when $n = 2$, and $\beta = (n-1)/n$ when $n \geq 3$. This choice of $\beta$ is taken for technical reasons to be seen later. Now we define
\begin{equation}\label{omegaformula}
\begin{aligned}
\omega_\epsilon := \sqrt{-1}\partial\overline{\partial}(|z|^2 + \gamma_1(z)\varphi(z) + \epsilon^2\gamma_2(z)\psi(\epsilon^{-1}z)) \text{ when } n \geq 3 \text{, and} \\
\omega_\epsilon := \sqrt{-1}\partial\overline{\partial}(|z|^2 + \gamma_1(z)\varphi(z) + \epsilon^2\gamma_2(z)\log|\epsilon^{-1}z|), \text{ when } n = 2.
\end{aligned}
\end{equation}

This is done in such a way so that $\omega_\epsilon \equiv \omega$ on $M \setminus B_{2r_\epsilon}$, and on $B_{r_\epsilon}$ we have that $\omega_\epsilon \equiv \epsilon^2\eta$ when we make the identification with $\tilde{B}_{\epsilon^{-1}r_\epsilon}$ in $Bl_0\mathbb C^n$. The K\"ahler class of $\omega_{\epsilon}$ is $\pi^*[\omega] - \epsilon^2[E]$.

\section{Weighted H\"older Spaces}
The geometry of the problem naturally breaks up the analysis into $M_p$ and $Bl_0\mathbb C^n$. It will be necessary for us to work in function spaces that measure growth and decay of a specified amount at the ends, these are the weighted H\"older spaces. Then on the total space of $Bl_pM$ we will work in a glued version of these functions spaces, introduced in \cite{Sz1} for the extremal metrics on blowups problem. For background on weighted H\"older spaces there is for example \cite{PacRiv}. We start with the case of $\mathbb R^n \setminus \{0\}$.

\begin{Definition}
    Let $k$ a nonnegative integer, $\alpha \in (0,1)$, and $\delta \in \mathbb R$. We define $C^{k,\alpha}(\mathbb R^n \setminus \{0\})$ to be the Banach space consisting of locally $C^{k,\alpha}$ functions $f: \mathbb R^n \setminus \{0\} \rightarrow \mathbb R$ for which the norm
    \[||f||_{C_\delta^{k,\alpha}(\mathbb R^n\setminus \{0\})} := \sup_{r > 0}||f_r||_{C^{k,\alpha}(B_2 \setminus B_1)} < \infty,\text{ where } f_r(x) := r^{-\delta}f(rx).\]
\end{Definition}
This norm seems opaque at first, but it is really just a slick way of requiring $\nabla^if = O(|x|^{\delta - i})$ at zero and infinity.

The primary operators of study in our problem are the linearization of scalar curvature at $\omega$ within its K\"ahler class, and its formal adjoint $L^*_\omega$. Their formulas are well known, first we have
\begin{equation}\label{Lformula}
L_\omega \phi := \frac{d}{dt}\bigg\rvert_{t = 0}S(\omega + t\sqrt{-1}\partial\overline{\partial}\phi) = -\frac{1}{4}\Delta^2_\omega \phi - \frac{1}{2}\text{Ric}_\omega\cdot \nabla_\omega^2\phi.
\end{equation}
Here, $\Delta_\omega = -d^*d$ denotes the full negative definite Riemannian Laplacian for $\omega$, and $\cdot$ denotes contraction with respect to the metric. For the formal adjoint,
\[L_\omega^*\phi = -\frac{1}{4}\Delta^2_\omega\phi - \frac{1}{4}\text{tr}_\omega\nabla_\omega(\text{tr}_\omega\nabla_\omega(\phi\text{Ric}_\omega)) = -\frac{1}{4}\Delta^2_\omega\phi - \frac{1}{2}\text{Ric}_\omega\cdot \nabla_\omega^2\phi - \partial \phi \cdot \overline{\partial}S(\omega) - \partial S(\omega) \cdot \overline{\partial}\phi - \frac{1}{2}\phi\Delta_\omega S(\omega),\] 
where $\cdot$ denotes contraction with respect to the metric, and $\text{tr}_\omega$ is trace taken with respect to the last two components, where components coming from covariant derivatives we list at the end. In coordinates, this means
\[\text{tr}_\omega\nabla_\omega(\phi \text{Ric}_\omega) = \nabla_{k}(\phi\text{Ric}\indices{_{\omega,j}^k}).\]

It is crucial to our argument to notice
\begin{equation}\label{diff}
(L^*_\omega - L_\omega)\phi = - \partial \phi \cdot \overline{\partial}S(\omega) - \partial S(\omega) \cdot \overline{\partial}\phi - \frac{1}{2}\phi\Delta_\omega S(\omega).
\end{equation}
In particular, $L^*_\omega$ and $L_\omega$ only differ by first order terms in $\phi$. Note that the highest order terms in $L_\omega$ and $L_\omega^*$ are the bilaplacian associated to $\omega$. Not only that, but on $M_p$ in the normal coordinate $z$ these terms approximate the usual Euclidean bilaplacian near $p$, and since $\eta$ is aymptotically Euclidean and scalar flat, $L_\eta = L^*_\eta$ approximates the Euclidean bilaplacian near infinity. Thus, we require results on the bilaplacian in these weighted spaces.

\begin{Proposition}\label{harm}
If $\delta \notin \mathbb Z \setminus (2 - n,0)$, the Euclidean Laplacian viewed as a map
\[\Delta_{euc}: C_\delta^{k,\alpha}(\mathbb R^n\setminus \{0\}) \rightarrow C^{k-2,\alpha}_{\delta - 2}(\mathbb R^n \setminus \{0\})\]
is an isomorphism. Thus, the same is true for the Euclidean bilaplacian $\Delta^2_{euc}: C^{k,\alpha}_\delta(\mathbb R^n \setminus \{0\}) \rightarrow C^{k-4,\alpha}_{\delta-4}(\mathbb R^n \setminus \{0\})$ for $\delta \notin \mathbb Z \setminus (4 - n,0)$.
\end{Proposition}
For a proof see for example \cite{PacRiv}. These discrete values for which the operators fail to be isomorphisms are known as \emph{indicial roots}. 

We now define the weighted spaces on $M_p$ and $Bl_0\mathbb C^n$. These are similar to the $\mathbb R^n \setminus \{0\}$ case, except these spaces only have one end. Note that given a function $f: M_p \rightarrow \mathbb R$, we can view $(1 - \gamma)f$ as a function on $\mathbb C^n \setminus \{0\}$ via the coordinate $z$ and extension by zero. Likewise, given a function $f: Bl_0\mathbb C^n \rightarrow \mathbb R$ we can view $\gamma f$ as a function on $\mathbb C^n \setminus \{0\}$. These spaces are also defined in \cite{AP1}.

\begin{Definition}  Let $k$ a nonnegative integer, $\alpha \in (0,1)$, and $\delta \in \mathbb R$.
    \begin{itemize}
    \item Define $C^{k,\alpha}_\delta(M_p)$ to be the Banach space of locally $C^{k,\alpha}$ functions $f: M_p \rightarrow \mathbb R$ such that
    \[||f||_{C_\delta^{k,\alpha}(M_p)} := ||f||_{C^{k,\alpha}_{\omega}(M \setminus B_1)} + ||(1 - \gamma)f||_{C_\delta^{k,\alpha}(\mathbb C^n \setminus \{0\})} < \infty.\]
    \item Define $C^{k,\alpha}_\delta(Bl_0\mathbb C^n)$ to be the Banach space of locally $C^{k,\alpha}$ functions $f: Bl_0\mathbb C^n \rightarrow \mathbb R$ such that
    \[||f||_{C_\delta^{k,\alpha}(Bl_0\mathbb C^n)} := ||f||_{C^{k,\alpha}_\eta(\tilde{B}_2)} + ||\gamma f||_{C_\delta^{k,\alpha}(\mathbb C^n \setminus \{0\})} < \infty.\]
    \end{itemize}
\end{Definition}

Finally, we can define the weighted spaces on $Bl_pM$ as they are defined in \cite{Sz1}. We will provide two equivalent norms. To this end, let $f: Bl_pM \rightarrow \mathbb R$, and $\delta \in \mathbb R$. For $r \in (\epsilon,1/2)$ we can define 
\[f_r: B_2\setminus B_1 \subset M \rightarrow \mathbb R\]
\[f_r(z) = r^{-\delta}f(rz),\]
and also define
\[f_\epsilon: \tilde{B}_1 \subset Bl_0\mathbb C^n \rightarrow \mathbb R\]
\[f_\epsilon(w) = \epsilon^{-\delta}f(z),\]
where again recall that $\tilde{B}_1 \subset Bl_0\mathbb C^n$ is identified with $B_\epsilon \subset Bl_pM$ via $z = \epsilon w$.

\begin{Definition}
    Let $k$ a nonnegative integer, $\alpha \in (0,1)$, and $\delta \in \mathbb R$. We let $C^{k,\alpha}_\delta(Bl_pM)$ denote the Banach space of locally $C^{k,\alpha}$ functions on $Bl_pM$ with the norm
    \[||f||_{C_\delta^{k,\alpha}(Bl_pM)} = ||f||_{C^{k,\alpha}_\omega(M \setminus B_1)} + \sup_{\epsilon < r < 1/2}r^{-\delta}||f||_{C^{k,\alpha}_{r^{-2}\omega_\epsilon}(B_{2r} \setminus B_r)} + \epsilon^{-\delta}||f||_{C^{k,\alpha}_{\epsilon^{-2}\omega_\epsilon}(B_\epsilon)}.\]
    An equivalent norm in $\epsilon$ is given by   
    \[||f||_{C_\delta^{k,\alpha}(Bl_pM)} = ||f||_{C^{k,\alpha}_\omega(M \setminus B_1)} + \sup_{\epsilon < r < 1/2}||f_r||_{C^{k,\alpha}(B_2 \setminus B_1)} + ||f_\epsilon||_{C^{k,\alpha}_\eta(\tilde{B}_1)},\]
    where in $B_2 \setminus B_1$ we are measuring the norms with the Euclidean metric.
\end{Definition}

The first definition can be convenient due to the fact that it is less coordinate dependent. The second is also convenient when working on computations solely in the normal coordinate $z$ in the gluing region. The first definition naturally extends to a weighted norm for tensors on $Bl_pM$. For the second definition, to extend the norm to tensors, by the middle term we will mean to measure the components of the tensor with respect to the normal coordinate $z$. Despite the notation $C^{k,\alpha}_\delta(Bl_pM)$, the norms also depend on $\epsilon$.

Note that the actual functions in these spaces never changes with $\delta$ or $\epsilon$, it always just consists of locally $C^{k,\alpha}$ functions on $Bl_pM$, only the norm changes with $\epsilon$ and $\delta$. Once again, upon first glance these norms are somewhat complicated, but the idea is that if $||f||_{C_\delta^{k,\alpha}(Bl_pM)} \leq c$, then
\[|\nabla^if| \leq c,\> |z| \geq 1\]
\[|\nabla^if| \leq c|z|^{\delta - i},\> \epsilon \leq |z| \leq 1\]
\[|\nabla^if| \leq c\epsilon^{\delta - i},\> |z| \leq \epsilon.\]

We end this section with various estimates on perturbations of $\omega_\epsilon$ and corresponding geometric quantities in the weighted norms that we will need later. Let $g_\epsilon$ denote the metric associated to $\omega_\epsilon$. These mostly appear in \cite{Sz1}.

\begin{Lemma}\label{PertEst}
For any nonnegative integer $k$, $\alpha \in (0,1)$ and $\delta \in \mathbb R$, there is a constant $c_0$ with the following property: there is a constant $C > 0$ independent of $\epsilon$ sufficiently small such that for any $||\phi||_{C_2^{k+4,\alpha}(Bl_pM)} < c_0$ we have $\omega_\phi:= \omega_\epsilon + \sqrt{-1}\partial\overline{\partial}\phi$ is positive and
\[||g_\phi - g_\epsilon||_{C_{\delta - 2}^{k+2,\alpha}(Bl_pM)},\> ||g_{\phi}^{-1} - g_\epsilon^{-1}||_{C_{\delta - 2}^{k+2,\alpha}(Bl_pM)}, ||Rm_{g_\phi} - Rm_{g_\epsilon}||_{C_{\delta - 4}^{k,\alpha}(Bl_pM)} < C||\phi||_{C_\delta^{k+4,\alpha}(Bl_pM)},\]
\[||L_{\omega_\phi} - L_{\omega_\epsilon}||_{C^{k+4,\alpha}_\delta \rightarrow C^{k,\alpha}_{\delta - 4}}, ||L^*_{\omega_\phi} - L^*_{\omega_\epsilon}||_{C^{k+4,\alpha}_\delta \rightarrow C^{k,\alpha}_{\delta - 4}} < C||\phi||_{C^{k+4,\alpha}_2(Bl_pM)},\]
where $g_\phi$ denotes the metric associated to $\omega_\phi$.
\end{Lemma}
\begin{proof}
Note that by the definition of the weighted norms we have
\begin{equation}\label{gep}
    ||g_\epsilon||_{C_0^{k,\alpha}(Bl_pM)} < C
\end{equation}
for a constant independent of $\epsilon$. The estimate on $g_\phi - g_\epsilon$ is clear by definition. We then have
\[g_\phi^{-1} - g_\epsilon^{-1} = g_\phi^{-1}(g_\epsilon - g_\phi)g_{\epsilon}^{-1}.\]
Then the desired estimate follows from
\[||g_\phi^{-1} - g_\epsilon^{-1}||_{C^{k+2,\alpha}_{\delta - 2}} \leq C||g_\phi^{-1}||_{C_0^{k + 2,\alpha}}||\sqrt{-1}\partial\overline{\partial}\phi||_{C_{\delta - 2}^{k+2,\alpha}}||g_\epsilon^{-1}||_{C_0^{k + 2,\alpha}}.\]
For the curvature estimate, if $g$ is a K\"ahler metric we have the formula
\[Rm_{g,i\overline{j}k\overline{\ell}} = -\partial_k\partial_{\overline{\ell}}g_{i\overline{j}} + g^{p\overline{q}}(\partial_kg_{i\overline{q}})(\partial_{\overline{\ell}}g_{p\overline{j}}).\]
From this formula we have
\[(Rm_{g_\phi} - Rm_{g_\epsilon})_{i\overline{j}k\overline{\ell}} = -\partial_k\partial_{\overline{\ell}}(g_{\phi,i\overline{j}} - g_{\epsilon,i\overline{j}}) + g_\phi^{p\overline{q}}(\partial_kg_{\phi,i\overline{q}})(\partial_{\overline{\ell}}g_{\phi,p\overline{j}}) - g_\epsilon^{p\overline{q}}(\partial_kg_{\epsilon,i\overline{q}})(\partial_{\overline{\ell}}g_{\epsilon,p\overline{j}}),\]
and then it follows directly from the estimates on the metrics. The estimates on the linearized operator also follows, for example by writing
\[\Delta^2_{\omega_\phi}\xi - \Delta^2_{\omega_\epsilon}\xi = g_\phi^{p\overline{q}}\partial_p\partial_{\overline{q}}([g^{k\overline{\ell}}_{\phi} - g^{k\overline{\ell}}_\epsilon]\partial_{k}\partial_{\overline{\ell}}\xi) + (g_\phi^{p\overline{q}} - g_\epsilon^{p\overline{q}})\partial_p\partial_{\overline{q}}(g_\epsilon^{k\overline{\ell}}\partial_k\partial_{\overline{\ell}}\xi),\]
and then we have
\[||\Delta^2_{\omega_\phi}\xi - \Delta^2_{\omega_\epsilon}\xi||_{C_{\delta - 4}^{k,\alpha}} \leq C(||g_\phi^{-1}||_{C^{k,\alpha}_0}||g_\phi^{-1} - g_\epsilon^{-1}||_{C^{k+2,\alpha}_0}||\xi||_{C^{k+4,\alpha}_\delta} + ||g_\phi^{-1} - g_\epsilon^{-1}||_{C^{k,\alpha}_0}||g_\epsilon^{-1}||_{C^{k+2,\alpha}_0}||\xi||_{C_\delta^{k+4,\alpha}})\]
We can similarly, using the bounds on the curvature tensor and the metric, get bounds on the Ricci and scalar curvatures.
\end{proof}

\section{Estimates on the Adjoint of the Linearized Operator}

The main observation in this section is that we can prove similar estimates on $L_{\omega_\epsilon}$ in $\epsilon$ as appear for example in \cite{Sz1}, \cite{SzB}, for the operator $L_{\omega_\epsilon}^*$. There are many approaches one can take in this step, and one can find several in the literature on blowing up extremal metrics. The one that we manage to make work in our our particular problem is the approach that appears in \cite{SzB}.

One encounters some extra analytic difficulties in $n = 2$ when proving estimates on $L_{\omega_\epsilon}$ in our setting, one of those being the usual weaker estimates on $L_{\omega_\epsilon}$ compared to what we can prove when $n \geq 3$, but the more problematic issue in $n = 2$ is that we are unable to prove the desired estimates on $L_{\omega_\epsilon}$ when the kernel of $L_\omega^*$ only consists of functions vanishing at the blow up point. This is the reason why we prove the estimates for $L_\omega^*$ instead of $L_\omega$, as $(L_\omega^*)^* = L_\omega$ always at least has constants in its kernel; this is used in step 1.5 of the proof of Proposition~\ref{invprop}. Note that this fact, together with the fact that $L_\omega$ is index zero, guarantees that $L_\omega^*$ always has a nontrivial kernel.

In general we know little about the kernel of $L_\omega^*$ when $\omega$ is not cscK. However, fix an $L^2$-orthonormal basis $f_1, ..., f_d$ for the kernel of $L_\omega$.

\begin{Lemma}\label{Minv}
    There exists points $q_1, ..., q_d \in M_p$ such that the operator
    \[\tilde{L}^*_\omega: C^{4,\alpha}(M) \rightarrow C^{0,\alpha}(M)\]
    \[\tilde{L}^*_\omega \phi := L_\omega^* \phi - \sum_{i = 1}^d\phi(q_i)f_i\]
    is an isomorphism.
\end{Lemma}
\begin{proof}
First, regardless of the choice of $q_1, ..., q_d$, if we have $\tilde{L}^*_\omega\phi = 0$, then we must have that $L_\omega^* \phi = 0$ since the $f_i$ are each orthogonal to the image of $L_\omega^*$. 

We claim that we can find $q_1, ..., q_d \in M_p$, and a basis $\phi_1, ..., \phi_d$ for $\ker(L_\omega^*)$ such that for $1 \leq i \leq d$ the vectors
\[v_i := (\phi_i(q_1), ..., \phi_i(q_d))\]
are linearly independent. If this were true, then suppose there was some nonzero $\phi$ in the kernel of $\tilde{L}^*_\omega$. Then $\phi$ is also in the kernel of $L_\omega^*$ so $\phi = \sum_{j = 1}^dc_j\phi_j$, and we would have
\[0 = \sum_{j = 1}^dc_j\phi_j(q_i)\]
for each $i$. This means there is a nonzero vector orthogonal to each $v_i^T$, which contradicts the linear independence of the $v_i$, so $\tilde{L}_\omega^*$ is injective. For surjectivity, $\tilde{L}_\omega^*$ differs from the index zero operator $L_{\omega}^*$ by an operator of finite rank, and so is also index zero. Then injectivity implies surjectivity.

Now we prove our claim. We do so inductively; first, if $\phi_1 \in \ker(L_\omega^*)$ is nonzero then of course we can find a point $q_1 \in M_p$ such that $\phi_1(q_1) \neq 0$. Now, suppose we have found points $q_1, ..., q_{d-1}$ and linear independent functions $\phi_1, ..., \phi_{d-1}$ in the kernel of $L_\omega^*$ such that for $1 \leq i \leq d - 1$ the vectors
\[v_i' := (\phi_i(q_1), ..., \phi_i(q_{d-1}))\]
are linearly independent. Let $\phi_d \in \ker(L_\omega^*)$ linearly independent of $\phi_1, ..., \phi_{d-1}$. Suppose there is $q_d \in M_p$ such that $\phi_d(q_d) \neq 0$ and $v_d$ is linearly dependent of $v_1, ..., v_{d-1}$ defined in the same way as above (they themselves are linearly independent due to the linear independence of the $v_i'$). Then we have that there exists constant $a_1, ..., a_{d-1}$ such that
\[\sum_{j = 1}^{d-1}a_j\phi_j(q_i) - \phi_d(q_i) = 0\> \text{ for } 1 \leq i \leq d,\> \sum_{j = 1}^{d-1}a_j\phi_j - \phi_d \not\equiv 0.\]
Then replacing $\phi_d$ with $\sum_{j = 1}^{d - 1}a_j\phi_j - \phi_d$ and replacing $q_d$ with a point where this function is nonzero yields the claim, since in that case $v_d$ would be equal to zero in its first $d - 1$ components.

If we could not find such a $q_d$, then $\phi_d(q_d) \neq 0$ implies that $v_1, ..., v_d$ are linearly independent as desired.
\end{proof}

From now on we fix $q_1, ..., q_d \in M_p$ to be a set of points for which the conclusion of the above lemma holds. When we blow up, we are pushed onto studying the operator
\[\tilde{L}^*_{\omega_\epsilon}: C^{4,\alpha}_\delta(Bl_pM) \rightarrow C^{0,\alpha}_\delta(Bl_pM)\]
\begin{equation}\label{Leptil}
\tilde{L}^*_{\omega_\epsilon} \phi := L^*_{\omega_\epsilon} \phi - \sum_{i = 1}^d\phi(q_i)f_i.
\end{equation}
Due to the geometry of $\omega_\epsilon$, on $M_p$ this operator behaves like $\tilde{L}^*_\omega$ restricted to $M_p$ so we need:

\begin{Proposition}\label{omegalin}
    Consider for now $\tilde{L}^*_\omega$ as an operator $\tilde{L}^*_\omega: C_\delta^{4,\alpha}(M_p) \rightarrow C^{0,\alpha}_{\delta - 4}(M_p)$.
    \begin{itemize}
        \item If $\delta > 4 - 2n$, then $\tilde{L}^*_{\omega}$ has trivial kernel.
        \item If $n = 2$, fix $\delta \in (-1,0)$. Then if $\phi \in C_\delta^{4,\alpha}(M_p)$ is in the kernel of $\tilde{L}^*_\omega$, in $B_1$ we have
        \[\phi = a\log|z| + b + \theta\]
        for some $a,b \in \mathbb R$ and $\theta \in C_{\delta + 2}^{4,\alpha}(M_p)$.
    \end{itemize}
\end{Proposition}
\begin{proof}
The proof of the first claim is exactly the same as what appears in \cite{AP1}, or \cite{SzB}. For the second claim, first note that since near $p$ we can write $\omega_{j\overline{k}} = \delta_{j\overline{k}} + O(|z|^2)$ as in \eqref{normal}, if $f \in C^{k,\alpha}_\delta(M_p)$ we have
\begin{equation}\label{deltaminus2}
\tilde{L}^*_\omega(f) + \frac{1}{4}\Delta^2_{euc}(f) \in C^{k-4,\alpha}_{\delta - 2}(M_p)
\end{equation}
as long as $\delta < 2$ (otherwise we have to account for the zeroth order terms in $\tilde{L}_\omega^*$). We would usually only expect $C^{k-4,\alpha}_{\delta - 4}(M_p)$ since these are fourth order operators.

Now consider $\phi \in C^{4,\alpha}_\delta(M_p)$ such that $\tilde{L}_\omega^* \phi = 0$. If we let $\chi$ be a cutoff supported in $B_1$ that is equal to $1$ in $B_{1/2}$, then from \eqref{deltaminus2}
\[\Delta^2_{euc}(\chi \phi) \in C^{0,\alpha}_{\delta - 2}(\mathbb C^2\setminus \{0\}).\]
From Proposition~\ref{harm} there is $\theta \in C^{4,\alpha}_{\delta + 2}(\mathbb C^2\setminus \{0\})$ such that $\theta - \chi \phi$ is biharmonic and in $C^{4,\alpha}_\delta(\mathbb C^2 \setminus \{0\})$. The only biharmonic functions with $\delta \in (-1,0)$ are multiples of $\log|z|$ and constants, and thus there are constants $a,b \in \mathbb R$ such that
\[\chi \phi = a\log|z| + b + \theta.\]
\end{proof}

We must also study $L_\eta$ since near the exceptional divisor $L^*_{\omega_\epsilon} = L^*_{\epsilon^2\eta} = L_{\epsilon^2\eta}$, where the last equality, once again, is due to the fact that $\eta$ is cscK.

\begin{Proposition}\label{etalin}
    If $\delta < 0$, the operator
    \[L_{\eta}: C_\delta^{4,\alpha}(Bl_0\mathbb C^n) \rightarrow C_{\delta - 4}^{0,\alpha}(Bl_0\mathbb C^n)\]
    has trivial kernel.
\end{Proposition}

The proof of the above appears for example in \cite{AP1} and \cite{SzB}. Finally, to prove the desired estimates on $\tilde{L}^*_{\omega_\epsilon}$ we need Schauder estimates which are independent of $\epsilon$.

\begin{Lemma}\label{Schauder}
    Fix $\delta < 0$. There exists a positive constant $C$ independent of all $\epsilon > 0$ small such that for all $\phi \in C^{4,\alpha}_\delta(Bl_pM)$
    \[||\phi||_{C^{4,\alpha}_\delta(Bl_pM)} \leq C(||\phi||_{C^0_\delta(Bl_pM)} + ||\tilde{L}^*_{\omega_\epsilon}\phi||_{C^{0,\alpha}_{\delta - 4}(Bl_pM)}).\]
\end{Lemma}
\begin{proof}
    First, we define a function $\rho: Bl_pM \rightarrow \mathbb R$ which describes the weights in the norms on $Bl_pM$,
    \[ \rho(z) = \begin{cases} 
      1 & |z| \geq 1 \\
      |z| & \epsilon \leq |z| \leq 1\\
      \epsilon & |z| \leq \epsilon
   \end{cases}
    .\]
    Now, $\eta$ is asymptotically Euclidean, $M$ is compact, and the sectional curvatures of $\omega_\epsilon$ satisfy a uniform bound in the $C^{k,\alpha}_{-2}(Bl_pM)$ norm \eqref{gep}, and thus we have that there exists some $r > 0$ independent of $\epsilon$ such that the injectivity radius of $\omega_\epsilon$ is bounded below by $r\rho(z)$. This implies that for any $\kappa > 0$ one can find fix an $r$ small enough independent of $\epsilon$ such that at each point $x \in Bl_pM$, in the normal coordinates of radius $r\rho(x)$ with respect to $\omega_\epsilon$ centered at $x$ we have
    \begin{equation}\label{holder_comp}\frac{1}{1 + \kappa}||\cdot||_{C^{k,\alpha}_{(r\rho(x))^{-2}\omega_\epsilon}} \leq ||\cdot||_{C^{k,\alpha}} \leq (1 + \kappa)||\cdot||_{C^{k,\alpha}_{(r\rho(x))^{-2}\omega_\epsilon}},
    \end{equation}
    where the middle norm is measured with respect to the Euclidean metric in the normal coordinates. Since $L^*_{\omega_{euc}} = -\frac{1}{4}\Delta^2_{euc}$, after possibly taking $r$ even smaller we further have for any $\phi$
    \[\frac{1}{1 + \kappa}||L^*_{(r\rho(x))^{-2}\omega_\epsilon}\phi||_{C^{k-4,\alpha}} \leq ||\frac{1}{4}\Delta^2_{euc}\phi||_{C^{k-4,\alpha}} \leq (1 + \kappa)||L^*_{(r\rho(x))^{-2}\omega_\epsilon}\phi||_{C^{k-4,\alpha}},\]
    again in the normal coordinate ball, and these norms are measured in the Euclidean norm with respect to the normal coordinates. Now, applying the usual Schauder estimates for the Euclidean bilaplacian, there is a constant $D > 0$ such that in any of these normal coordinate balls
    \[||\phi||_{C^{k,\alpha}} \leq D(||\phi||_{C^0} + ||\frac{1}{4}\Delta^2_{euc}\phi||_{C^{k-4,\alpha}}).\]
    Due to the comparisons of the Euclidean H\"older norms with the H\"older norms for $\omega_\epsilon$ \eqref{holder_comp} we then have
    \[||\phi||_{C^{k,\alpha}_{(r\rho(x))^{-2}\omega_\epsilon}} \leq D(1 + \kappa)^2(||\phi||_{C^0} + ||L^*_{(r\rho(x))^{-2}\omega_\epsilon}\phi||_{C^{k-4,\alpha}}).\]
    $D(1 + \kappa)^2$ has no dependence on $\epsilon$. Noting that $L^*_{(r\rho(x))^{-2}\omega_\epsilon} = (r\rho(x))^4L^*_{\omega_\epsilon}$ and applying this estimate to $\rho(x)^{-\delta}\phi$ over all of these normal coordinate balls we have a $C > 0$ independent of $\epsilon$ such that
    \[||\phi||_{C^{4,\alpha}_\delta(Bl_pM)} \leq C(||\phi||_{C^0_\delta(Bl_pM)} + ||L^*_{\omega_\epsilon}\phi||_{C^{0,\alpha}_{\delta - 4}(Bl_pM)}).\]
    Now to get the estimate for $\tilde{L}^*_{\omega_\epsilon}$, note that there is $c > 0$ independent of $\epsilon$ such that
    \[\left |||\tilde{L}^*_{\omega_\epsilon}\phi||_{C^{0,\alpha}_{\delta - 4}(Bl_pM)} - ||L^*_{\omega_\epsilon}\phi||_{C^{0,\alpha}_{\delta - 4}(Bl_pM)}\right| \leq c||\phi||_{C^0_\delta(Bl_pM)}\max_i||f_i||_{C^{0,\alpha}_{\delta - 4}(Bl_pM)}.\]
    The result then follows from the fact that the norms $||f_i||_{C^{0,\alpha}_{\delta - 4}(Bl_pM)}$ are uniformly bounded from above in $\epsilon$ when $\delta < 0$.
\end{proof}

We are now ready to prove the main estimate.

\begin{Proposition}\label{invprop}
    Fix $\delta \in (-1,0)$ sufficiently small. Consider the operator
    \[\tilde{L}_{\omega_\epsilon}^*: C_\delta^{4,\alpha}(Bl_pM) \rightarrow C^{0,\alpha}_{\delta - 4}(Bl_pM).\]
    \begin{itemize}
        \item Suppose $n > 2$. For $\epsilon > 0$ small there exists a constant $K > 0$ independent of $\epsilon$ such that for all $\phi \in C_{\delta}^{4,\alpha}(Bl_pM)$,
        \[||\phi||_{C_{\delta}^{4,\alpha}(Bl_pM)} \leq K||\tilde{L}^*_{\omega_\epsilon}\phi||_{C_{\delta - 4}^{0,\alpha}(Bl_pM)}.\]
        \item Suppose $n = 2$. For $\epsilon > 0$ small there exists a constant $K > 0$ independent of $\epsilon$ such that for all $\phi \in C_{\delta}^{4,\alpha}(Bl_pM)$,
        \[||\phi||_{C_{\delta}^{4,\alpha}(Bl_pM)} \leq K\epsilon^\delta||\tilde{L}^*_{\omega_\epsilon}\phi||_{C_{\delta - 4}^{0,\alpha}(Bl_pM)}.\]
    \end{itemize}
\end{Proposition}
\begin{proof}
    We will prove both cases simultaneously and comment where the differences occur. We argue by contradiction: suppose there exists a sequence of functions $\phi_i$ such that
    \begin{equation}\label{contra1}
    ||\phi_i||_{C_\delta^{4,\alpha}(Bl_pM)} = 1,\> ||\tilde{L}^*_{\omega_{\epsilon_i}}\phi_i||_{C^{0,\alpha}_{\delta - 4}(Bl_pM)} \leq \frac{1}{i},
    \end{equation}
    in the $n > 2$ case, and in the $n = 2$ case we assume furthermore there is $\epsilon_i \rightarrow 0$ such that
    \begin{equation}\label{contra2}
    ||\phi_i||_{C_\delta^{4,\alpha}(Bl_pM)} = 1,\> ||\tilde{L}^*_{\omega_{\epsilon_i}}\phi_i||_{C^{0,\alpha}_{\delta - 4}(Bl_pM)} \leq \frac{\epsilon_i^{-\delta}}{i}.
    \end{equation}
    From the uniform Schauder estimates Lemma~\ref{Schauder}, we have up to choosing a subsequence $\phi_i \rightarrow \phi_\infty$ locally in $C^{4,\alpha}$ on $M_p$ with $\phi_\infty \in C_\delta^{4,\alpha}(M_p)$, where we may have had to replace $\alpha$ which a slightly smaller constant.
    
    \underline{Step 1}: We show $\phi_\infty \equiv 0$ on $M_p$. When $n > 2$ this immediately follows from the fact that $\tilde{L}_{\omega}^*\phi_\infty = 0$ on $M_p$ and our choice of $\delta$ due to Proposition ~\ref{omegalin}. When $n = 2$, our slightly stronger assumption \eqref{contra2} gives for each $1 \leq j \leq d$
    \begin{equation}\label{intbound}
    \left|\int_{Bl_pM}f_j\tilde{L}^*_{\omega_{\epsilon_i}}(\phi_i)\frac{\omega_{\epsilon_i}^2}{2}\right| = \left|\int_{Bl_pM}f_jL^*_{\omega_{\epsilon_i}}(\phi_i)\frac{\omega_{\epsilon_i}^2}{2} - \sum_{k = 1}^d\phi_i(q_k)\int_{Bl_pM}f_jf_k\frac{\omega_{\epsilon_i}^2}{2}\right| < \frac{C}{i},
    \end{equation}
    where $C$ is a constant independent of $\epsilon$. First note that since $f_1, ..., f_d$ were assumed to be $L_2$-orthonormal that
    \[\lim_{i \rightarrow \infty}\int_{Bl_pM}f_jf_k\frac{\omega_{\epsilon_i}^2}{2} = \delta_{jk},\]
    and on the other hand we have
    \[\int_{Bl_pM}f_jL^*_{\omega_{\epsilon_i}}(\phi_i)\frac{\omega_{\epsilon_i}^2}{2} = \int_{Bl_pM}L_{\omega_{\epsilon_i}}(f_j)\phi_i\frac{\omega_{\epsilon_i}^2}{2}.\]
    Now we need to analyze this term as $i \rightarrow \infty$. We break up our manifold into $M\setminus B_{2r_{\epsilon_i}}$, $B_{2r_{\epsilon_i}} \setminus B_{r_{\epsilon_i}}$, and $B_{r_{\epsilon_i}}$, recalling that $r_\epsilon = \epsilon^\beta$ as defined in section 2. For $M \setminus B_{2r_{\epsilon_i}}$, in this region $\omega_{\epsilon_i} = \omega$ so we have
    \[\int_{M \setminus B_{2r_{\epsilon_i}}}L_{\omega_{\epsilon_i}}(f_j)\phi_i\frac{\omega_{\epsilon_i}^2}{2} = \int_{M \setminus B_{2r_{\epsilon_i}}}L_\omega(f_j)\phi_i\frac{\omega^2}{2} = 0.\]
    Next, in $B_{r_{\epsilon_i}}$ after the coordinate change $z = \epsilon_i w$ we have $\omega_{\epsilon_i} = \epsilon_i^2\eta$.  Note that $L_{\epsilon^2\eta} = \frac{1}{\epsilon^4}L_\eta$, so from \eqref{Lformula} we have
    \[\int_{B_{r_{\epsilon_i}}}L_{\omega_{\epsilon_i}}(f_j)\phi_i\frac{\omega_{\epsilon_i}^2}{2} = \int_{\tilde{B}_{r_{\epsilon_i}\epsilon_i^{-1}}}-\left(\frac{1}{4}\Delta^2_{\eta}(f_j(\epsilon_i w)) + \frac{1}{2}\text{Ric}_{\eta}\cdot \nabla_\eta^2(f_j(\epsilon_i w))\right)\phi_i(\epsilon_i w)\frac{\eta^2}{2}\]
    \[= \int_{\tilde{B}_{r_{\epsilon_i}\epsilon_i^{-1}}} -[\eta^{k\overline{\ell}}\partial_k\partial_{\overline{\ell}}(\eta^{a\overline{b}}\partial_a\partial_{\overline{b}}(f_j(\epsilon_i w))) + \text{Ric}_{\eta}^{k\overline{\ell}}\partial_{k}\partial_{\overline{\ell}}(f_j(\epsilon_i w))]\phi_i(\epsilon_i w)\frac{\eta^2}{2}.\]
    We also have that from \eqref{omega2} $\eta^{k\overline{\ell}} = \delta^{k\overline{\ell}} + O(|w|^{-2})$, $\text{Ric}_\eta^{k\overline{\ell}} = O(|w|^{-4})$, $|\nabla^k(f_j(\epsilon w))| \leq C\epsilon^k$, and $|\phi_i(\epsilon w)| = |\phi_i(z)| \leq C\epsilon^\delta$, where $C$ represents possibly different constants but nevertheless is independent of $\epsilon$. Thus,
    \[\left|\int_{B_{r_{\epsilon_i}}}L_{\omega_{\epsilon_i}}(f_j)\phi_i\frac{\omega_{\epsilon_i}^2}{2}\right| \leq C\left(\epsilon_i^4 + \epsilon_i^2\right)\epsilon^\delta\text{Vol}_{\eta}(\tilde{B}_{r_{\epsilon_i}\epsilon_i^{-1}}) \leq C\epsilon_i^{2}\epsilon_i^\delta r_{\epsilon_i}^4\epsilon_i^{-4} = C\epsilon_i^{4\beta - 2 + \delta},\]
    so this goes to zero as $i \rightarrow \infty$ as long as $\beta > 1/2$ and $\delta$ is taken small enough. Lastly, we look at the integral over $B_{2r_{\epsilon_i}} \setminus B_{r_{\epsilon_i}}$. Here we have
    \[\int_{B_{2r_{\epsilon_i}}\setminus B_{r_{\epsilon_i}}}L_{\omega_{\epsilon_i}}(f_j)\phi_i\frac{\omega_{\epsilon_i}^2}{2}=\int_{B_{2r_{\epsilon_i}}\setminus B_{r_{\epsilon_i}}}-\left(\frac{1}{4}\Delta^2_{\omega_{\epsilon_i}}f_j + \frac{1}{2}\text{Ric}_{\omega_{\epsilon_i}}\cdot \nabla_{\omega_{\epsilon_i}}^2f_j\right)\phi_i\frac{\omega_{\epsilon_i}^2}{2}.\]
    From the explicit formula for $\omega_{\epsilon_i}$ \eqref{omegaformula}, in the region $r_{\epsilon_i} \leq |z| \leq 2r_{\epsilon_i}$ we have
    \[\omega_{\epsilon_i}^{k\overline{\ell}} = \delta^{k\overline{\ell}} + O\left(\frac{\epsilon_i^2}{r_{\epsilon_i}^2}\log|\epsilon_i|\right),\>\text{Ric}^{k\overline{\ell}}_{\omega_{\epsilon_i}} = O\left(\frac{\epsilon_i^2}{r_{\epsilon_i}^4}\log|\epsilon_i|\right),\]
    and thus
    \[\left|\int_{B_{2r_{\epsilon_i}}\setminus B_{r_{\epsilon_i}}}L_{\omega_{\epsilon_i}}(f_j)\phi_i\frac{\omega_{\epsilon_i}^2}{2}\right| \leq C\left(1 + \frac{\epsilon_i^2}{r_{\epsilon_i}^4}\log|\epsilon_i|\right)\epsilon^\delta\text{Vol}_{\omega_{\epsilon_i}}(B_{2r_{\epsilon_i}}\setminus B_{r_{\epsilon_i}})\]
    \[\leq C\epsilon_i^2r_{\epsilon_i}^{-4}\log|\epsilon_i|\epsilon_i^\delta r_{\epsilon_i}^4 = C\epsilon_i^{2+ \delta}\log|\epsilon_i|.\]
    This goes to zero as $i \rightarrow \infty$ for $\delta$ small. The upshot is that from \eqref{intbound} we then get
    \[\lim_{i \rightarrow \infty}\left|\int_{Bl_pM}f_j\tilde{L}^*_{\omega_{\epsilon_i}}(\phi_i)\frac{\omega_{\epsilon_i}^2}{2}\right| = |\phi_\infty(q_j)| = 0,\]
    and thus on $M_p$, we not only have $\tilde{L}_\omega^* \phi_\infty = 0$ but
    \[L_\omega^* \phi_\infty = 0.\]
    We then have that in $B_1$ there is $a,b \in \mathbb R$ and $\theta \in C_{\delta + 2}^{4,\alpha}(M_p)$ such that $\phi_\infty = a\log|z| + b + \theta$ from Proposition~\ref{omegalin}.

    \underline{Step 1.5}: Continuing our focus on the $n = 2$ case, we now show that there is a constant $\kappa \neq 0$ such that for any $\xi \in C^\infty(M)$,
    \[\int_M L_\omega(\xi)\phi_\infty \frac{\omega^2}{2} = -a\kappa\xi(p).\]
    Thus, testing this equation with $\xi \equiv 1$, we must have that $a = 0$. Therefore $\phi_\infty$ extends smoothly to all of $M$, being a bounded weak solution to an elliptic equation on $M$. Then $\phi_\infty \equiv 0$ by the invertibility of $\tilde{L}^*_\omega$ on $M$ (Lemma~\ref{Minv}), as desired. This is the key step that necessitates us proving estimates on $L^*_{\omega_\epsilon}$ instead of $L_{\omega_{\epsilon}}$. Anyway, we write
    \[\int_M L_\omega(\xi)\phi_\infty \frac{\omega^2}{2} = \lim_{r \rightarrow 0}\int_{M \setminus B_r}-\left(\frac{1}{4}\Delta^2_{\omega}\xi + \frac{1}{2}\text{Ric}_{\omega}\cdot \nabla_{\omega}^2\xi\right)\phi_\infty \frac{\omega^2}{2}.\]
    Now we apply integration by parts to each term in the integrand. Denoting by $N$ the unit normal to the boundary sphere $\partial B_r$, we first look at
    \[\int_{M \setminus B_r}\Delta^2_\omega(\xi)\phi_\infty\frac{\omega^2}{2} = \int_{M \setminus B_r} \xi\Delta^2_\omega (\phi_\infty)\frac{\omega^2}{2} - \int_{\partial B_r} \xi\nabla_N(\Delta_\omega \phi_\infty)\text{Vol}_{\partial B_r} + \int_{\partial B_r}\Delta_\omega(\phi_\infty)\nabla_N(\xi)\text{Vol}_{\partial B_r}\]
    \[ - \int_{\partial B_r}\Delta_\omega(\xi)\nabla_N(\phi_\infty)\text{Vol}_{\partial B_r} + \int_{\partial B_r}\phi_\infty\nabla_N(\Delta_\omega \xi)\text{Vol}_{\partial B_r}.\]
    We investigate the boundary terms as $r \rightarrow 0$. Since $\phi_\infty$ behaves like $a\log|z|$ near zero, we have that the last three terms all go to zero in $r$ since the terms involving the derivatives of $\phi_\infty$ blow up at most like $O(r^{-2})$, whereas the volume of $\partial B_r$ decays like $O(r^3)$. The only issue is the first boundary term. When $r$ is small we have 
    \[\xi\nabla_N(\Delta_\omega \phi_\infty) = a\xi \nabla_N[(\delta^{k\overline{\ell}} + O(r^2))\partial_k\partial_{\overline{\ell}}\log|z|] + \xi\nabla_N(\Delta_\omega \theta).\]
    Due to the estimate on $\theta$, the last term decays at worst like $O(r^{\delta -1})$, and so integrating and taking the limit we have
    \[\lim_{r \rightarrow 0}\int_{M \setminus B_r}-\frac{1}{4}\Delta^2_\omega (\xi)\phi_\infty \frac{\omega^2}{2} = \int_{M}-\frac{1}{4}\Delta^2_\omega (\phi_\infty) \xi\frac{\omega^2}{2} + \lim_{r \rightarrow 0}\frac{a}{4}\int_{\partial B_r}\xi\nabla_N(\Delta_{euc}\log|z|)\text{Vol}_{\partial B_r}\]
    \[= \int_{M}-\frac{1}{4}\Delta^2_\omega (\phi_\infty) \xi\frac{\omega^2}{2} + a\kappa\xi(p).\]
    For the other integration by parts, we have
    \[\int_{M \setminus B_r}(\text{Ric}_{\omega}\cdot\nabla_\omega^2\xi)\phi_\infty\frac{\omega^2}{2}\]
    \[= \int_{M \setminus B_r}\xi\text{tr}_\omega\nabla_\omega(\text{tr}_\omega\nabla_\omega (\phi_\infty\text{Ric}_\omega))\frac{\omega^2}{2} + \int_{\partial B_r}\langle \nabla \xi \otimes N^b, \phi_\infty\text{Ric}_\omega \rangle \text{Vol}_{\partial B_r} - \int_{\partial B_r}\xi (\text{tr}_\omega \nabla (\phi_\infty\text{Ric}_{\omega}))(N)\text{Vol}_{\partial B_r}.\]
    Both of the boundary terms have integrands which blow up no worse than $O(r^{-1})$, so both boundary terms go to zero as $r\rightarrow 0$. Thus, we finally have that
    \[\int_ML_\omega (\xi)\phi_\infty \frac{\omega^2}{2} = \int_M \xi L^*_\omega(\phi_\infty)\frac{\omega^2}{2} - a\kappa\xi(p) = -a\kappa\xi(p),\]
    since $\phi_\infty$ is in the kernel of $L^*_\omega$ on $M_p$.
    
    \underline{Step 2}: The rest of the proof is now the same in any dimension and proceeds as in \cite{SzB}. To recount, we have as in \eqref{contra1} $\phi_i$ such that
\[||\phi_i||_{C_\delta^{4,\alpha}(Bl_pM)} = 1,\> ||\tilde{L}^*_{\omega_{\epsilon_i}}(\phi_i)||_{C_{\delta - 4}^{0,\alpha}(Bl_pM)} < \frac{1}{i},\]
(the second bound is actually slightly stronger in $n = 2$) and $\phi_i \rightarrow 0$ locally in $C^{4,\alpha}$, where again we may have to replace $\alpha$ with a slightly smaller constant. We then have the uniform Schauder estimate of Lemma~\ref{Schauder}
\[||\phi_i||_{C_\delta^{4,\alpha}(Bl_pM)} \leq C(||\phi_i||_{C_\delta^0(Bl_pM)} + ||\tilde{L}^*_{\omega_{\epsilon_i}}(\phi_i)||_{C^{0,\alpha}_{\delta - 4}(Bl_pM)}),\]
so then we have
\[1 \leq C||\phi_i||_{C^0_\delta(Bl_pM)} + \frac{C}{i} \iff \frac{1}{C} - \frac{1}{i} \leq ||\phi_i||_{C_\delta^0(Bl_pM)} \leq 1,\]
and so $||\phi_i||_{C_\delta^0(Bl_pM)}$ is uniformly bounded below by a positive constant, say $C'$. Applying the uniform Schauder estimates to $\psi_i := \frac{\phi_i}{||\phi_i||_{C_\delta^0(Bl_pM)}}$ we get
\[||\psi_i||_{C_\delta^{4,\alpha}(Bl_pM)} \leq C + \frac{C}{i||\phi_i||_{C_\delta^0(Bl_pM)}} \leq C + \frac{C}{iC'}.\]
Since the last term goes to zero we have a family of functions such that 
\begin{equation}\label{concen1}
    ||\psi_i||_{C_\delta^0(Bl_pM)} = 1,\> ||\psi_i||_{C_\delta^{4,\alpha}(Bl_pM)} \leq C,
\end{equation}
\begin{equation}\label{concen2}
||L^*_{\omega_{\epsilon_i}}(\psi_i)||_{C_{\delta - 4}^{0,\alpha}(Bl_pM)} \rightarrow 0,\> \psi_i \rightarrow 0 \text{ locally in } C^{4,\alpha} \text{ on } M_p.
\end{equation}
Now, we want to look at the point $x_i \in Bl_pM$ where $\rho_i^{-\delta}\psi_i$ achieves its maximum, where $\rho_i$ is as defined in the proof of Lemma~\ref{Schauder} but with respect to $\epsilon_i$, which due to the above assumptions is when
\[\rho_i^{-\delta}(x_i)\psi_i(x_i) = 1.\]
Since $\psi_i \rightarrow 0$ locally on $M_p$, $\rho_i(x_i) \rightarrow 0$. Then $\epsilon_i^{-1}\rho_i(x_i)$ is either bounded or unbounded; one can think of the former case as the $\psi_i$ concentrating close to the exceptional divisor as $i \rightarrow \infty$, the latter case is when $\psi_i$ concentrates in the Euclidean neck region.

\underline{Step 3}: We first suppose the former. In this case we have that $\epsilon_i^{-1}\rho_i(x_i) < R$ for some $R$ and for all $i$. Then under the identification of $B_1 \subset Bl_pM$ with the ball $\tilde{B}_{\epsilon^{-1}} \subset Bl_0\mathbb C^n$, we have that all the $x_i$ are inside $\tilde{B}_R$. Thus, up to taking a subsequence we have $x_i \rightarrow x_\infty \in \tilde{B}_R$. Rescaling by $\epsilon_i$, $\epsilon_i^{-\delta}\psi_i$ can be viewed as functions on larger subsets of $Bl_0\mathbb C^n$ still with a uniform $C_\delta^{4,\alpha}(Bl_0\mathbb C^n)$ bound from \eqref{concen1}. Thus we have that up to a subsequence and replacing $\alpha$, $\epsilon_i^{-\delta}\psi_i \rightarrow \psi_\infty$ locally in $C^{4,\alpha}$ on $Bl_0\mathbb C^n$.

Now note that on $Bl_0\mathbb C^n$ we have that $\epsilon_i^{-2}\omega_{\epsilon_i}$ locally converges to the Burns-Simanca metric $\eta$ and thus since
\[\epsilon_i^{-\delta}||L_{\omega_{\epsilon_i}}^*\psi_i||_{C_{\delta - 4}^{0,\alpha}(Bl_pM)} \rightarrow 0 \text{ from }\eqref{concen2} \text{, and }L_{\epsilon_i^{-2}\omega_{\epsilon_i}} = \epsilon_i^4L_{\omega_{\epsilon_i}},\]
we have $L_{\eta}\psi_\infty = 0$ on $Bl_0\mathbb C^2$. On the other hand by assumption
\[\epsilon_i^{-\delta}\rho_i(x_i)^\delta > R^\delta,\]
and by definition of $x_i$
\[\psi_i(x_i) = \rho_i(x_i)^\delta,\]
so
\[\psi_\infty(x_\infty) \geq R^\delta.\]
Thus, we have a contradiction of Proposition~\ref{etalin} since $\psi_\infty$ decays at infinity but is nonzero.

\underline{Step 4}: Lastly, we need to deal with the case when $\epsilon_i^{-1}\rho_i(x_i)$ is unbounded. Note again that $\rho_i(x_i) \rightarrow 0$, and thus in this case we have $\rho_i(x_i) = |x_i|$. Choose $r_i \rightarrow 0$ and $R_i \rightarrow \infty$ such that
\[R_i|x_i| \rightarrow 0,\> r_i|x_i|\epsilon_i^{-1} \rightarrow \infty.\]
We consider the annulus $A_i := B_{R_i|x_i|} \setminus B_{r_i|x_i|}$. By rescaling by $|x_i|$, we identify $A_i$ with $B_{R_i}\setminus B_{r_i} \subset \mathbb C^n \setminus \{0\}$. On $B_{R_i}\setminus B_{r_i}$ we consider the metric $|x_i|^{-2}\omega_{\epsilon_i}$ (here $\omega_{\epsilon_i}$ on $B_{R_i} \setminus B_{r_i}$ is $\omega_{\epsilon_i}$ on $A_i$ pulled back along the rescaling by $|x_i|$). The point of these annuli is that on $M_p$ they are shrinking to $p$, where $\omega$ is becoming euclidean, and on $Bl_0\mathbb C^n$ they are going to infinity where the Burns-Simanca metric is Euclidean. In other words, these metrics converge locally in any $C^k$ norm to the standard Euclidean metric on $\mathbb C^n \setminus \{0\}$. 

Now, under this identification all of the points $x_i$ lie on the unit sphere in $\mathbb C^n \setminus \{0\}$, and thus up to a subsequence $x_i \rightarrow x_\infty$. Further, the functions $|x_i|^{-\delta}\psi_i$, when viewed as functions on larger annuli in $\mathbb C^n \setminus \{0\}$, inherit uniform $C_\delta^{4,\alpha}$ bounds from \eqref{concen1}. Thus, up to a sequence $|x_i|^{-\delta}\psi_i \rightarrow \psi_\infty$ locally in $C^{4,\alpha}$ for possibly smaller $\alpha$. We then have $\psi_\infty \in C_\delta^{4,\alpha}$ and
\[\psi_\infty(x_\infty) = 1,\> \Delta_{euc}^2\psi_\infty = 0,\]
a contradiction of Proposition~\ref{harm} based on our choice of $\delta$.
\end{proof}

The above proposition can be summarized as follows.

\begin{Corollary}\label{Coro}
    Let $\delta \in (-1,0)$ sufficiently small.
    \begin{itemize}
        \item Suppose $n > 2$. For $\epsilon$ small there exists a constant $K > 0$ independent of $\epsilon$ such that $\tilde{L}^*_{\omega_\epsilon}$ is an isomorphism and
        \[||(\tilde{L}^*_{\omega_\epsilon})^{-1}||_{C^{4,\alpha}_\delta \rightarrow C_{\delta - 4}^{0,\alpha}} \leq K.\]
        \item Suppose $n = 2$. For $\epsilon$ small there exists a constant $K > 0$ independent of $\epsilon$ such that $\tilde{L}^*_{\omega_\epsilon}$ is an isomorphism and
        \[||(\tilde{L}^*_{\omega_\epsilon})^{-1}||_{C^{4,\alpha}_\delta \rightarrow C_{\delta - 4}^{0,\alpha}} \leq K\epsilon^\delta.\]
    \end{itemize}
\end{Corollary}
\begin{proof}
    The above proposition shows that $\tilde{L}^*_{\omega_\epsilon}$ is injective. The operator $\tilde{L}_{\omega_\epsilon}^*$ differs from the index zero operator $L^*_{\omega_\epsilon}$ by a finite rank operator, and thus must also have index zero. Thus we also have surjectivity, and the statement about the norms follows directly.
\end{proof}

We wish to draw attention to the fact that we could have proven the exact same estimates but for the operator $L_{\omega_\epsilon}$ in $n \geq 3$. This makes the remainder of the proof of the main theorem slightly cleaner, we only insist on $L^*_{\omega_\epsilon}$ to uniformize the approach across dimensions. The precise reason that we are unable to prove the estimates for $L_{\omega_\epsilon}$ in $n = 2$ is due to the added analytic difficulty necessitating the existence of a function on $M$ in the kernel of $L^*_\omega$ which is nonvanishing at the blowup point. Since our problem is open in nature (in the topological sense), one could ask if this condition is satisfied for K\"ahler metrics near our initial metric $\omega$, perhaps even within its own K\"ahler class. Similar generic aspects of eigenfunctions have been studied for example by K. Uhlenbeck \cite{Uhl} in the obviously less restrictive Riemannian case, it would be interesting to study similar questions in the K\"ahler case.

\section{An Improvement on $\omega_\epsilon$ when $n = 2$}
For this section we assume $n = 2$ exclusively. We will construct instead a metric $\tilde{\omega}_\epsilon$ on $Bl_pM$ which is a slight alteration of $\omega_\epsilon$ but still is a scaling down of the Burns-Simanca metric in $B_{r_\epsilon}$. The main advantage will be that the gluing annulus $B_{2r_\epsilon}\setminus B_{r_\epsilon}$ will have better behavior, namely we will no longer have to multiply the $\log$ term in the definition of $\omega_\epsilon$ by a cutoff. This better approximation is necessitated by the need to prove some estimate on the difference $S(\omega_\epsilon) - S(\omega)$. What appears here is only a minor alteration to the construction in \cite{SzB}.

 The main point is to argue that there exists a function $g \in \text{Ker}(L_\omega^*)$ on all of $M$, and a function $\Gamma$ on $M_p$ such that
\begin{equation}\label{g}
    L_\omega \Gamma = g
\end{equation}
on $M_p$, and $\Gamma(z) = \log|z| + \psi(z)$ near $p$, where $\psi \in C_1^{4,\alpha}(M_p)$. Note that by elliptic regularity $g$ must be smooth on $M$, and $\Gamma$ is also smooth on $M_p$.

To this end, let $\chi$ once again be a cutoff supported in $B_2$ and equal to $1$ in $B_1$. Let $\chi_r := \chi(z/r)$ for any $r > 0$. Then we will define the metric
    \[\omega'_r := \sqrt{-1}\partial\overline{\partial}(|z|^2 + \chi_r(z)\varphi(z))\]
    on $\mathbb C^2 \setminus \{0\}$. This metric is flat outside $B_{2r}$ and equal to $\omega$ in $B_r$. We need the following estimates.
    \begin{Lemma}
    \begin{itemize}
        \item For any $\delta \in \mathbb R$ and sufficiently small $r > 0$ there is a positive constant $C$ independent of $r$ such that
        \[||(L_{\omega'_r} + \frac{1}{4}\Delta^2_{euc})f||_{C_{\delta -4}^{0,\alpha}(\mathbb C^2 \setminus \{0\})} \leq Cr^2||f||_{C_\delta^{4,\alpha}(\mathbb C^2\setminus \{0\})}\]
        for any $f: \mathbb C^2\setminus \{0\} \rightarrow \mathbb R$. Then for $r$ sufficiently small $L_{\omega'_r}$ is an isomorphism whenever $\delta$ is not an indicial root of $\Delta^2_{euc}$.
        \item For any $\delta \leq -2$ we have $L_{\omega'_r}(\log|z|) \in C^{0,\alpha}_{\delta}(B_1)$.
    \end{itemize}
    \end{Lemma}
    \begin{proof}
      For the first claim, outside of $B_{2r}$ we already have that $L_{\tilde{\omega}_r} = -\frac{1}{4}\Delta^2_{euc}$, so we only focus within $B_{2r}$. Note that within $B_{2r}$ we have
      \[\nabla^i(\chi_r\varphi) = O(|z|^{4-i}),\]
      and the metric is given by
    \[\omega'_{r,j\overline{k}} = \delta_{j\overline{k}} + \partial_j\partial_{\overline{k}}(\chi_r\varphi).\]
    As long as $r$ is sufficiently small, in $B_{2r}$ we have
    \[\omega_r^{'j\overline{k}} = \delta^{j\overline{k}} - \partial_j\partial_{\overline{k}}(\chi_r\varphi) + O(|z|^4).\]
    For the Ricci curvature in $B_{2r}$ we have
    \[R_{j\overline{k}} = \partial_j\partial_{\overline{k}}(\log\det [\delta_{\ell\overline{m}} + \partial_\ell\partial_{\overline{m}}(\chi_r\varphi)]) = \partial_j\partial_{\overline{k}}(\partial_\ell\partial_{\overline{\ell}}(\chi_r\varphi)) + O(|z|^2),\]
    so then we have that in $B_{2r}$
    \[L_{\omega'_r}f = -\frac{1}{4}\Delta^2_{euc}f + \delta^{\ell\overline{m}}\partial_\ell\partial_{\overline{m}}[\partial_j\partial_{\overline{k}}(\chi_r\varphi)\partial_j\partial_{\overline{k}}f] + \partial_j\partial_{\overline{k}}(\partial_\ell\partial_{\overline{\ell}}(\chi_r\varphi))\partial_j\partial_{\overline{k}}f + O(|z|).\]
    Thus we have in $B_{2r}$ there is some constant $C$ such that
    \[|z|^2\left(\frac{1}{4}\Delta^2_{euc}f + L_{\omega'_r}f\right) = |z|^2(\delta^{\ell\overline{m}}\partial_\ell\partial_{\overline{m}}[\partial_j\partial_{\overline{k}}(\chi_r\varphi)\partial_j\partial_{\overline{k}}f] + \partial_j\partial_{\overline{k}}(\partial_\ell\partial_{\overline{\ell}}(\chi_r\varphi))\partial_j\partial_{\overline{k}}f + O(|z|))\]
    \[\leq Cr^2\partial_j\partial_{\overline{k}}f\]
    for $r$ small, which proves the desired estimate.

    For the second claim we have
    \[L_{\omega'_r}\log |z| = -\frac{1}{4}\Delta^2_{euc}\log|z| + \delta^{\ell\overline{m}}\partial_{\ell}\partial_{\overline{m}}[\partial_j\partial_{\overline{k}}(\chi_r\varphi)\partial_j\partial_{\overline{k}}\log|z|] + \partial_j\partial_{\overline{k}}(\partial_\ell\partial_{\overline{\ell}}(\chi_r\varphi))\partial_j\partial_{\overline{k}}\log|z| + O(|z|^{-1}).\]
    The claim then follows since $\log|z|$ is biharmonic away from zero.
    \end{proof}
    From this lemma we have that for any $r$ sufficiently small there exists $\xi_1 \in C^{4,\alpha}_{2-\tau}$ for any $\tau > 0$ such that 
    \[L_{\omega'_r}(\xi_1) = L_{\omega'_r}(\log|z|),\]
    which implies that near $p$
    \[L_\omega(\chi_r(z)\log|z| - \chi_r(z)\xi_1(z)) = 0,\]
    and the left side is a function on all of $M$. Arguing in the exact same fashion as in Lemma~\ref{Minv}, if $g_1, ..., g_d$ is an $L^2$-orthonormal basis for $\text{Ker}(L_\omega^*)$ then there exists points $q_1', ..., q_d' \in M_p$ such that there is a function $\xi_2 \in C^{4,\alpha}(M)$ satisfying
    \[L_\omega(\chi_r\log|z| - \chi_r\xi_1 - \xi_2) = -\sum_{i = 1}^d \xi_2(q_i')g_i\]
    on $M_p$. We can subtract a constant from $\xi_2$ to ensure that $\xi_2(p) = 0$, and then we set 
    \[\Gamma = \chi_r(\log|z| - \xi_1) - \xi_2,\]
    which then has the desired properties.

    Now we define
    \begin{equation}\label{omegatilde1}
    \tilde{\omega}_\epsilon := \omega_\epsilon + \epsilon^2\sqrt{-1}\partial\overline{\partial}(\gamma_1(\Gamma - \log\>\epsilon)).
    \end{equation}
    We have that near $p$, $\Gamma(z) = \log|z| + \psi$ with $\psi \in C_1^{4,\alpha}(M_p)$, so we can alternatively write
    \begin{equation}\label{omegatilde}
    \tilde{\omega}_\epsilon = \sqrt{-1}\partial\overline{\partial}(|z|^2 + \gamma_1(z)\varphi(z) + \epsilon^2\log|z| + \epsilon^2\gamma_1(z)\psi(z)).
    \end{equation}
    Note that in this presentation it is clear that we have preserved the property that the metric is equal to $\epsilon^2\eta$ in $B_{r_\epsilon}$ after changing coordinates.
    
    Lastly, we will need to have the same bounds on $(\tilde{L}^*_{\tilde{\omega}_\epsilon})^{-1}$ as we did for $(\tilde{L}^*_{\omega_\epsilon})^{-1}$. We have them because $||\gamma_1\Gamma||_{C_2^{4,\alpha}(Bl_pM)} \leq C\frac{\log|r_\epsilon|}{r_\epsilon^2}$, 
    and using Lemma~\ref{PertEst} in the last inequality we have
    \[||\phi||_{C_{\delta}^{4,\alpha}(Bl_pM)} \leq K\epsilon^{\delta}||\tilde{L}^*_{\omega_\epsilon}\phi||_{C^{0,\alpha}_{\delta-4}(Bl_pM)} \leq K\epsilon^\delta(||\tilde{L}^*_{\tilde{\omega}_\epsilon}\phi||_{C^{0,\alpha}_{\delta-4}(Bl_pM)} + ||(L^*_{\omega_\epsilon} - L^*_{\tilde{\omega}_\epsilon})\phi||_{C^{0,\alpha}_{\delta-4}(Bl_pM)})\]
    \[\leq K\epsilon^\delta(||\tilde{L}^*_{\tilde{\omega}_\epsilon}\phi||_{C^{0,\alpha}_{\delta-4}(Bl_pM)} + C(\epsilon^{2(1 - \beta)}\beta\log|\epsilon|)||\phi||_{C_{\delta}^{4,\alpha}(Bl_pM)}),\]
    which is equivalent to
    \begin{equation}\label{tildeinv}
    ||\phi||_{C_{\delta}^{4,\alpha}(Bl_pM)} \leq \frac{K\epsilon^\delta}{1 -  KC\epsilon^{2(1 - \beta) + \delta}\beta\log|\epsilon|}||\tilde{L}^*_{\tilde{\omega}_\epsilon}\phi||_{C^{0,\alpha}_{\delta-4}(Bl_pM)} \leq (K + 1)\epsilon^\delta ||\tilde{L}^*_{\tilde{\omega}_\epsilon}\phi||_{C^{0,\alpha}_{\delta-4}(Bl_pM)},
    \end{equation}
    for small $\epsilon$, given that $\delta \in (-1,0)$. We will just replace $K + 1$ with $K$ when discussing this bound later.
\section{Proof of the Main Theorem}

Per \eqref{form}, we now use the estimates on $(\tilde{L}_{\omega_\epsilon}^*)^{-1}$ and $(\tilde{L}_{\tilde{\omega}_\epsilon}^*)^{-1}$ to find and solve an equation with right hand side going to zero as $\epsilon \rightarrow 0$, providing the main theorem.  We deal with the $n > 2$ case first.

\subsection{The Equation for $n > 2$}We would like to find $\phi$ solving the equation
\[S(\omega) - S(\omega_\epsilon + \sqrt{-1}\partial\overline{\partial}\phi)= -\partial \phi \cdot \overline{\partial}S(\omega_\epsilon) - \partial S(\omega_\epsilon) \cdot \overline{\partial}\phi - \frac{1}{2}\phi\Delta_{\omega_\epsilon} S(\omega_\epsilon) - \sum_{i = 1}^d\phi(q_i)f_i,\>\> (*_1)\]
for each $\epsilon$ sufficiently small, where the $f_i$ and $q_i$ are as in Lemma~\ref{Minv}. Linearizing, we can instead write $S(\omega_\epsilon + \sqrt{-1}\partial\overline{\partial}\phi)$ as
\[S(\omega_\epsilon + \sqrt{-1}\partial\overline{\partial}\phi) = S(\omega_\epsilon) + L_{\omega_\epsilon}(\phi) + Q_{\omega_\epsilon}(\phi),\]
where $Q_{\omega_\epsilon}(\phi)$ are the higher order terms. Then the equation $(*_1)$ is equivalent to
\[S(\omega) - S(\omega_\epsilon) - Q_{\omega_\epsilon}(\phi) = L_{\omega_\epsilon}\phi -\partial \phi \cdot \overline{\partial}S(\omega_\epsilon) - \partial S(\omega_\epsilon) \cdot \overline{\partial}\phi - \frac{1}{2}\phi\Delta_{\omega_\epsilon} S(\omega_\epsilon) - \sum_{i = 1}^d\phi(q_i)f_i.\]
From the formula \eqref{diff} and the definition \eqref{Leptil} we recognize the right side of the equation now as $\tilde{L}_{\omega_\epsilon}^*\phi$, multiplying both sides by the inverse we have that $(*_1)$ is equivalent to
\[(\tilde{L}_{\omega_\epsilon}^*)^{-1}(S(\omega) - S(\omega_\epsilon) - Q_{\omega_\epsilon}(\phi)) = \phi.\]
Now, if we consider the map
\[\mathcal{N}_\epsilon: C_\delta^{4,\alpha}(Bl_pM) \rightarrow C_{\delta}^{4,\alpha}(Bl_pM)\]
\[\phi \mapsto (\tilde{L}_{\omega_\epsilon}^*)^{-1}(S(\omega) - S(\omega_\epsilon) - Q_{\omega_\epsilon}(\phi)),\]
$(*_1)$ can be rephrased as a fixed point problem for the nonlinear operator $\mathcal{N}_\epsilon$. We will now work to apply the Banach fixed point theorem. We already have bounds on $(\tilde{L}^*_{\omega_\epsilon})^{-1}$ from Corollary~\ref{Coro}, we also need bounds on $S(\omega_\epsilon) - S(\omega)$ and $Q_{\omega_\epsilon}(\phi)$.

\begin{Lemma}\label{Q}
    Let $k$ be a nonnegative integer, $\alpha \in (0,1)$, and $\delta \in \mathbb R$. Then if $c_0$ is as in Lemma~\ref{PertEst} and
    \[||\phi_1||_{C_2^{k+4,\alpha}(Bl_pM)},||\phi_2||_{C_2^{k+4,\alpha}(Bl_pM)} \leq c_0,\]
    then there is some $C > 0$ independent of $\epsilon$ such that
    \[||Q_{\omega_\epsilon}(\phi_2) - Q_{\omega_\epsilon}(\phi_1)||_{C^{k,\alpha}_{\delta}(Bl_pM)} \leq C(||\phi_1||_{C_2^{k+4,\alpha}(Bl_pM)} + ||\phi_2||_{C_2^{k+4,\alpha}(Bl_pM)})||\phi_1 - \phi_2||_{C_\delta^{k+4,\alpha}(Bl_pM)},\]
    for all $\epsilon$ small.
\end{Lemma}
\begin{proof}
    By the mean value inequality there is a $t \in [0,1]$ such that $\chi = t\phi_1 + (1 - t)\phi_2$ satisfies
\[||Q_{\omega_\epsilon}(\phi_2) - Q_{\omega_\epsilon}(\phi_1)||_{C^{k,\alpha}_{\delta}(Bl_pM)} \leq ||DQ_{\omega_\epsilon,\chi}(\phi_2 - \phi_1)||_{C^{k,\alpha}_{\delta}(Bl_pM)}.\]
On the other hand we have
\[DQ_{\omega_\epsilon,\chi} = L_{\omega_\chi} - L_{\omega_\epsilon},\]
where $\omega_\chi := \omega_\epsilon + \sqrt{-1}\partial\overline{\partial}\chi$. Then since $||\chi||_{C_{2}^{k+4,\alpha}(Bl_pM)} \leq c_0$ by Lemma~\ref{PertEst} we have
\[||Q_{\omega_\epsilon}(\phi_2) - Q_{\omega_\epsilon}(\phi_1)||_{C^{k,\alpha}_{\delta}(Bl_pM)} \leq C||\chi||_{C_{2}^{k+4,\alpha}(Bl_pM)}||\phi_1 - \phi_2||_{C_\delta^{k+4,\alpha}(Bl_pM)}\]
\[\leq C(||\phi_1||_{C_2^{k+4,\alpha}(Bl_pM)} + ||\phi_2||_{C_2^{k+4,\alpha}(Bl_pM)})||\phi_1 - \phi_2||_{C_\delta^{k+4,\alpha}(Bl_pM)}.\]
\end{proof}

\begin{Lemma}\label{nonlin1}
Let $\delta < 0$ be chosen sufficiently small. There is a constant $c_1 > 0$ independent of $\epsilon$ such that if 
\[||\phi_1||_{C_2^{4,\alpha}(Bl_pM)},||\phi_2||_{C_2^{4,\alpha}(Bl_pM)} \leq c_1,\]
then
\[||\mathcal{N}_\epsilon(\phi_1) - \mathcal{N}_\epsilon(\phi_2)||_{C^{4,\alpha}_\delta(Bl_pM)} \leq \frac{1}{2}||\phi_1 - \phi_2||_{C_\delta^{4,\alpha}(Bl_pM)},\]
for all $\epsilon$ small.
\end{Lemma}
\begin{proof}
First, we have that
\[\mathcal{N}_\epsilon(\phi_1) - \mathcal{N}_\epsilon(\phi_2) = (\tilde{L}_{\omega_\epsilon}^*)^{-1}(Q_{\omega_\epsilon}(\phi_2) - Q_{\omega_\epsilon}(\phi_1)).\]
From the bound on $(\tilde{L}_{\omega_\epsilon}^*)^{-1}$ in Corollary~\ref{Coro} and the bounds in Lemma~\ref{Q}, the result follows once $c_1$ is chosen small enough.
\end{proof}

\begin{Lemma}\label{ScalarEst1}
    For $\epsilon$ sufficiently small we have
    \[||S(\omega_\epsilon) - S(\omega)||_{C_{\delta - 4}^{2,\alpha}(Bl_pM)} \leq Cr_\epsilon^{4 - \delta}\]
    for some constant $C$ independent of $\epsilon$.
\end{Lemma}
\begin{proof}
$S(\omega_\epsilon) \equiv S(\omega)$ on $M \setminus B_{2r_\epsilon}$ and $S(\omega_\epsilon) \equiv 0$ on $B_{r_\epsilon}$, so the lemma really only needs to be checked in $B_{2r_\epsilon} \setminus B_{r_\epsilon}$. Rewriting \eqref{omegaformula} in $n \geq 3$, in this region the key is that
\[\omega_\epsilon = \omega_{euc} + \sqrt{-1}\partial\overline{\partial}(\gamma_1(z)\varphi(z) + \epsilon^2\gamma_2(z)\psi(\epsilon^{-1}z)),\]
and for any $\delta$
\[||\gamma_1(z)\varphi(z) + \epsilon^2\gamma_2(z)\psi(\epsilon^{-1}z)||_{C_\delta^{6,\alpha}(B_{2r_\epsilon}\setminus B_{r_\epsilon})} \leq Cr_\epsilon^{4 - \delta}.\]
This estimate is because when $\epsilon$ is small, in the annulus $B_{2r_\epsilon} \setminus B_{r_\epsilon}$ we have
\[\epsilon^2\psi(\epsilon^{-1}z) = O\left(\epsilon^2\frac{r_\epsilon^{4 - 2n}}{\epsilon^{4 - 2n}}\right) = O(r_\epsilon^4),\]
specifically due to our choice of $\beta$. With this estimate we can conclude
\[||S(\omega_\epsilon)||_{C_{\delta - 4}^{2.\alpha}(B_{2r_\epsilon} \setminus B_{r_\epsilon})} \leq Cr_\epsilon^{4 - \delta}.\]

\end{proof}

Combining all the estimates, we can solve $(*_1)$.

\begin{Proposition}
    Assume $n > 2$ and $\delta < 0$ sufficiently small. Then for all $\epsilon > 0$ sufficiently small, $(*_1)$ admits a solution with $||\phi||_{C^{4,\alpha}_\delta(Bl_pM)} \leq c_1\epsilon^{2-\delta}$.
\end{Proposition}
\begin{proof}
Let $\mathcal{B}$ denote the ball of radius $c_1\epsilon^{2 - \delta}$ in $C_\delta^{4,\alpha}(Bl_pM)$. If $\phi \in \mathcal{B}$, then $||\phi||_{C_2^{4,\alpha}(Bl_pM)} \leq c_1$. Then the above estimates give that $\mathcal{N}_\epsilon$ is a contraction on $\mathcal{B}$ for $\epsilon$ small. We show in addition that $\mathcal{N}_\epsilon(\mathcal{B}) \subset \mathcal{B}$ for $\epsilon$ small, from which the proposition follows immediately by the Banach fixed point theorem. 

First, we have from Lemma~\ref{nonlin1}
\[||\mathcal{N}_\epsilon(\phi)||_{C_\delta^{4,\alpha}(Bl_pM)} \leq ||\mathcal{N}_\epsilon(\phi) - \mathcal{N}_\epsilon(0)||_{C_\delta^{4,\alpha}(Bl_pM)} + ||\mathcal{N}_\epsilon(0)||_{C_\delta^{4,\alpha}(Bl_pM)}\]
\[\leq \frac{1}{2}||\phi||_{C_\delta^{4,\alpha}(Bl_pM)} + ||\mathcal{N}_\epsilon(0)||_{C_\delta^{4,\alpha}(Bl_pM)}.\]
Recall $\mathcal{N}_\epsilon(0) = (\tilde{L}^*_{\omega_\epsilon})^{-1}(S(\omega) - S(\omega_\epsilon))$, so from the uniform estimates on $(\tilde{L}^*_{\omega_\epsilon})^{-1}$ in Corollary~\ref{Coro} and the estimate on $S(\omega) - S(\omega_\epsilon)$ in Lemma~\ref{ScalarEst1} we have independent of $\epsilon$ the bound
\[||\mathcal{N}_\epsilon(0)||_{C_\delta^{4,\alpha}(Bl_pM)} \leq Cr_\epsilon^{4 - \delta}.\]
Now, $(4 - \delta)\beta > 2 - \delta$ when $n > 2$ and $\delta$ is chosen small enough, so then when $\epsilon$ is sufficiently small
\[||\mathcal{N}_\epsilon(0)||_{C_\delta^{4,\alpha}(Bl_pM)} \leq \frac{1}{2}c_1\epsilon^{2 - \delta},\]
which completes the proof.
\end{proof}

\subsection{The Equation for $n = 2$} The $n = 2$ case is only very slightly different. We instead solve the equation
\[S(\omega) - S(\tilde{\omega}_\epsilon + \sqrt{-1}\partial\overline{\partial}\phi) = -\partial\phi \cdot \overline{\partial}S(\tilde{\omega}_\epsilon) - \partial S(\tilde{\omega}_\epsilon)\cdot\overline{\partial} \phi - \frac{1}{2}\phi\Delta_{\tilde{\omega}_\epsilon}S(\tilde{\omega}_\epsilon) - \sum_{i = 1}^d\phi(q_i)f_i - \epsilon^2g,\>\> (*_2)\]
where we have the addition of $g$ from \eqref{g}. By an entirely similar calculation as above, $(*_2)$ is equivalent to
\[(\tilde{L}^*_{\tilde{\omega}_{\epsilon}})^{-1}(S(\omega) - S(\tilde{\omega}_\epsilon) + \epsilon^2g - Q_{\tilde{\omega}_\epsilon}(\phi)) = \phi,\]
so we instead consider the fixed point problem for the operator
\[\mathcal{N}_\epsilon: C_\delta^{4,\alpha}(Bl_pM) \rightarrow C_{\delta}^{4,\alpha}(Bl_pM)\]
\[\phi \mapsto (\tilde{L}_{\tilde{\omega}_\epsilon}^*)^{-1}\left(S(\omega) - S(\tilde{\omega}_\epsilon) + \epsilon^2g - Q_{\tilde{\omega}_\epsilon}(\phi)\right).\]
Once again we just need to prove similar estimates as in the $n > 2$ case.
\begin{Lemma}\label{N2}
Let $\delta < 0$ be chosen sufficiently small. There is a constant $c_1 > 0$ independent of $\epsilon$ such that if 
\[||\phi_1||_{C_2^{4,\alpha}(Bl_pM)},||\phi_2||_{C_2^{4,\alpha}(Bl_pM)} \leq c_1\epsilon^{-\delta},\]
then
\[||\mathcal{N}_\epsilon(\phi_1) - \mathcal{N}_\epsilon(\phi_2)||_{C^{4,\alpha}_\delta(Bl_pM)} \leq \frac{1}{2}||\phi_1 - \phi_2||_{C_\delta^{4,\alpha}(Bl_pM)},\]
for all $\epsilon$ small.
\end{Lemma}
\begin{proof}
The proof is identical to Lemma~\ref{nonlin1}, we just need $\epsilon^{-\delta}$ to account for the slightly weaker bound on $(\tilde{L}^*_{\tilde{\omega}_\epsilon})^{-1}$.
\end{proof} 

\begin{Lemma}\label{SC2}
    Let $\delta < 0$ be chosen sufficiently small. For $\epsilon$ sufficiently small we have
    \[||S(\tilde{\omega}_\epsilon) - \epsilon^2 g - S(\omega)||_{C^{2,\alpha}_{\delta - 4}} \leq Cr_{\epsilon}^{4 - \delta}\]
    for some constant $C$ independent of $\epsilon$.
\end{Lemma}
\begin{proof}
In $B_{r_\epsilon}$ there is nothing to prove since there $\tilde{\omega}_\epsilon$ is scalar flat. In $B_{2r_\epsilon}\setminus B_{r_\epsilon}$ we have from rewriting \eqref{omegatilde}
\[\tilde{\omega}_\epsilon = \omega_{euc} + \sqrt{-1}\partial \overline{\partial}(\epsilon^2\log|z| + \gamma_1(\varphi + \epsilon^2\psi)).\]
Since $L_{\omega_{euc}} = -\frac{1}{4}\Delta^2_{euc}$ \eqref{Lformula} we have
\[S(\tilde{\omega}_\epsilon) = - \frac{1}{4}\Delta^2_{euc}(\gamma_1(\varphi + \epsilon^2\psi)) + Q_{\omega_{euc}}(\epsilon^2\log|z| + \gamma_1(\varphi + \epsilon^2\psi)).\]
Now, for the $Q_{\omega_{euc}}$ term we have, recalling that $\psi \in C_1^{4,\alpha}(M_p)$
\[||Q_{\omega_{euc}}(\epsilon^2\log|z| + \gamma_1(\varphi + \epsilon^2\psi))||_{C^{2,\alpha}_{\delta - 4}(B_{2r_\epsilon}\setminus B_{r_\epsilon})}\]
\[\leq C||\epsilon^2\log|z| + \gamma_1(\varphi + \epsilon^2\psi)||_{C^{6,\alpha}_2(B_{2r_\epsilon}\setminus B_{r_\epsilon})}||\epsilon^2\log|z| + \gamma_1(\varphi + \epsilon^2\psi)||_{C^{6,\alpha}_\delta(B_{2r_\epsilon}\setminus B_{r_\epsilon})}\]
\[\leq C(\epsilon^{2 - 2\beta}\log|\epsilon| + \epsilon^{2\beta})(\epsilon^{2 - \delta}\log|\epsilon| + \epsilon^{(4 - \delta)\beta}) \leq r_{\epsilon}^{4 - \delta},\]
where we have used Lemma~\ref{Q}, and thus
\[||S(\tilde{\omega}_\epsilon)||_{C_{\delta - 4}^{2,\alpha}(B_{2r_\epsilon} \setminus B_{r_\epsilon})} \leq Cr_{\epsilon}^{4 - \delta}.\]

On $M \setminus B_{2r_\epsilon}$ we have from \eqref{omegatilde1} $\tilde{\omega}_\epsilon = \omega + \epsilon^2\sqrt{-1}\partial\overline{\partial}\Gamma$ and thus from \eqref{g}
\[S(\tilde{\omega}_\epsilon) = S(\omega) + \epsilon^2L_\omega\Gamma + Q_\omega(\epsilon^2\Gamma) = S(\omega) + \epsilon^2g + Q_{\omega_{\epsilon}}(\epsilon^2\Gamma).\]
We can estimate the last term using Lemma~\ref{Q} once again to get
\[||Q_{\omega_\epsilon}(\epsilon^2\Gamma)||_{C^{2,\alpha}_{\delta - 4}(M\setminus B_{2r_\epsilon})} \leq C||\epsilon^2\Gamma||_{C_2^{6,\alpha}(M\setminus B_{2r_\epsilon})}||\epsilon^2\Gamma||_{C_\delta^{6,\alpha}(M\setminus B_{2r_\epsilon})}\leq C\frac{\epsilon^2}{r_\epsilon^2}\log|\epsilon|\frac{\epsilon^2}{r_\epsilon^\delta}\log
|\epsilon| \leq Cr_{\epsilon}^{4 - \delta}\]
due to our range of $\beta$. Putting these regions together gives the desired estimate.

\end{proof}

Combining all the estimates, we can solve $(*_2)$.

\begin{Proposition}\label{phiest}
    Assume $n = 2$ and $\delta < 0$ sufficiently small. Then for all $\epsilon > 0$ sufficiently small, $(*_2)$ admits a solution with $||\phi||_{C_\delta^{4,\alpha}(Bl_pM)} \leq c_1\epsilon^{2 - 2\delta}$.
\end{Proposition}
\begin{proof}
Let $\mathcal{B}$ denote the ball of radius $c_1\epsilon^{2 - 2\delta}$ in $C_\delta^{4,\alpha}(Bl_pM)$. If $\phi \in \mathcal{B}$, then $||\phi||_{C_2^{4,\alpha}(Bl_pM)} \leq c_1\epsilon^{-\delta}$. Then the above estimates give that $\mathcal{N}_\epsilon$ is a contraction on $\mathcal{B}$ for $\epsilon$ small. Again, we show that $\mathcal{N}_\epsilon(\mathcal{B}) \subset \mathcal{B}$ for $\epsilon$ small, and we conclude by the Banach fixed point theorem. 

First, we have using Lemma~\ref{N2}
\[||\mathcal{N}_\epsilon(\phi)||_{C_\delta^{4,\alpha}(Bl_pM)} \leq ||\mathcal{N}_\epsilon(\phi) - \mathcal{N}_\epsilon(0)||_{C_\delta^{4,\alpha}(Bl_pM)} + ||\mathcal{N}_\epsilon(0)||_{C_\delta^{4,\alpha}(Bl_pM)}\]
\[\leq \frac{1}{2}||\phi||_{C_\delta^{4,\alpha}(Bl_pM)} + ||\mathcal{N}_\epsilon(0)||_{C_\delta^{4,\alpha}(Bl_pM)}.\]
Recall $\mathcal{N}_\epsilon(0) = (\tilde{L}^*_{\tilde{\omega}_\epsilon})^{-1}\left(S(\omega) - S(\tilde{\omega}_\epsilon) + \epsilon^2g\right)$, so from the uniform estimates on $(\tilde{L}^*_{\tilde{\omega}_\epsilon})^{-1}$ from \eqref{tildeinv} and the estimates in Lemma~\ref{SC2} on $S(\omega) - S(\tilde{\omega}_\epsilon) + \epsilon^2g$ we have independent of $\epsilon$ the bound
\[||\mathcal{N}_\epsilon(0)||_{C_\delta^{4,\alpha}(Bl_pM)} \leq C\epsilon^\delta r_\epsilon^{4 - \delta}.\]
Now, $\delta + (4 - \delta)\beta > 2 - 2\delta$ for our range of $\beta$ if $\delta$ is chosen small enough, so then when $\epsilon$ is sufficiently small
\[||\mathcal{N}_\epsilon(0)||_{C_\delta^{4,\alpha}(Bl_pM)} \leq \frac{1}{2}c_1\epsilon^{2 - 2\delta},\]
which completes the proof.
\end{proof}

\subsection{Proof of Theorem~\ref{mainthm}} Once again, the point is that both $(*_1)$ and $(*_2)$ are equations for the difference $S(\omega) - S(\omega_\epsilon + \sqrt{-1}\partial\overline{\partial}\phi)$, and we can solve them with sufficiently good estimates on $\phi$ and $S(\omega_\epsilon)$ and their derivatives, so that the right hand side of the equation becomes small in $\epsilon$. This gives us our main theorem.

\begin{proof}{(of Theorem~\ref{mainthm})}
    Let's do the $n = 2$ case, the $n > 2$ case is entirely similar, just without the $g$ term. We can get that $(*_2)$ has a solution with $||\phi||_{C^{4,\alpha}_\delta(Bl_pM)} \leq c_1\epsilon^{2 - 2\delta}$ for all $\epsilon$ sufficiently small (Proposition ~\ref{phiest}). If we can show that
    \[\max_M |-\partial\phi \cdot \overline{\partial}S(\tilde{\omega}_\epsilon) - \partial S(\tilde{\omega}_\epsilon)\cdot\overline{\partial} \phi - \frac{1}{2}\phi\Delta_{\tilde{\omega}_\epsilon}S(\tilde{\omega}_\epsilon) - \sum_{i = 1}^d\phi(q_i)f_i - \epsilon^2g| \rightarrow 0, \text{ as } \epsilon \rightarrow 0,\]
    then $S(\omega_\epsilon + \sqrt{-1}\partial\overline{\partial}\phi)$ will have the same sign everywhere as $S(\omega)$. Since $M$ is compact, $g$ and the $f_i$ are bounded, and $|\phi| \leq c_1\epsilon^2$ on all of $M$, so we do not need to be concerned about the last two terms. For the first two terms, note first of all that $S(\tilde{\omega}_\epsilon) \equiv 0$ in $B_{r_\epsilon}$, so in that region we are fine. In $M \setminus B_{r_\epsilon}$ we have
    \[|\nabla \phi| \leq c_1r_\epsilon^{\delta - 1}\epsilon^{2 - 2\delta} = c_1\epsilon^{2 - \beta + (\beta - 2)\delta},\]
    \[|\nabla [S(\omega_\epsilon) - \epsilon^2g - S(\omega)]| \leq Cr_{\epsilon}^{\delta - 5}r_{\epsilon}^{4 - \delta} = C\epsilon^{-\beta},\]
    and $\epsilon^{2(1 - \beta) + (\beta - 2)\delta} \rightarrow 0$ as $\epsilon \rightarrow 0$ for our range of $\beta$, so the terms $-\partial\phi \cdot \overline{\partial}S(\tilde{\omega}_\epsilon)- \partial S(\tilde{\omega}_\epsilon)\cdot\overline{\partial} \phi$ are fine. What remains is the term $ - \frac{1}{2}\phi\Delta_{\tilde{\omega}_\epsilon}S(\tilde{\omega}_\epsilon)$. For this note once again that in $M \setminus B_{r_\epsilon}$ we have
    \[|\phi| \leq c_1r_{\epsilon}^\delta\epsilon^{2 - 2\delta} = c_1\epsilon^{2 + (\beta - 2)\delta},\]
    \[|\nabla^2 [S(\omega_\epsilon) - \epsilon^2g - S(\omega)]| \leq Cr_{\epsilon}^{\delta - 6}r_{\epsilon}^{4 - \delta} = C\epsilon^{-2\beta},\]
    and $\epsilon^{2(1 - \beta) + (\beta - 2)\delta} \rightarrow 0$ as $\epsilon \rightarrow 0$ for our range of $\beta$.

    Lastly, we need to check that any solution $\phi$ to $(*_1)$ or $(*_2)$ will be smooth, so that the metric we construct on $Bl_pM$ with our desired scalar curvature properties is indeed smooth. This is a routine elliptic bootstrapping argument but we will write the details anyway. Let $\phi \in C^{4,\alpha}(Bl_pM)$ be a solution to $(*_1)$ or $(*_2)$. Then in particular, in either case for some $h \in C^{\infty}(M)$,
    \[S(\tilde{\omega}_\epsilon + \sqrt{-1}\partial\overline{\partial}\phi) = \partial\phi \cdot \overline{\partial}S(\tilde{\omega}_\epsilon) + \partial S(\tilde{\omega}_\epsilon)\cdot\overline{\partial} \phi + \frac{1}{2}\phi\Delta_{\tilde{\omega}_\epsilon}S(\tilde{\omega}_\epsilon) + h.\]
    We can also split this into a system of equations
    \[\frac{1}{2}\Delta_{\tilde{\omega}_\phi}F = \partial\phi \cdot \overline{\partial}S(\tilde{\omega}_\epsilon) + \partial S(\tilde{\omega}_\epsilon)\cdot\overline{\partial} \phi + \frac{1}{2}\phi\Delta_{\tilde{\omega}_\epsilon}S(\tilde{\omega}_\epsilon) + h,\]
    \[\log \det(\tilde{g}_{\phi, j\overline{k}}) = F,\]
    where for convenience we have written $\tilde{\omega}_\phi := \tilde{\omega}_\epsilon + \sqrt{-1}\partial\overline{\partial}\phi$, and for the associated metric $\tilde{g}_\phi$. For the first equation, we see that $F$ solves a linear second order elliptic PDE with coefficients in $C^{2,\alpha}(Bl_pM)$, and the right hand side is in $C^{3,\alpha}(Bl_pM)$, so we have that $F \in C^{4,\alpha}(Bl_pM)$. Differentiating both sides of the second equation we get
    \[(\tilde{g}_\phi)^{j \overline{k}}\partial_j\partial_{\overline{k}}(\partial_i\phi) = -(\tilde{g}_\phi)^{j \overline{k}}(\partial_i\tilde{g}_{\epsilon,j\overline{k}}) + \partial_iF,\]
    so $\partial_i\phi$ can be thought of as solving a linear second order elliptic PDE with coefficients in $C^{2,\alpha}(Bl_pM)$ and right hand side in $C^{2,\alpha}(Bl_pM)$. Since we already had that $\partial_i\phi \in C^{3,\alpha}(Bl_pM)$ this yields that $\partial_i\phi \in C^{4,\alpha}(Bl_pM)$, and finally $\phi \in C^{5,\alpha}(Bl_pM)$. Continuing this argument inductively gives that $\phi$ is actually smooth.    
\end{proof}

Here we state the special case for positive (and negative) scalar curvature K\"ahler metrics, which was the original motivation for this paper.

\begin{Theorem}\label{psccor}
    Let $(M,\omega)$ be a compact K\"ahler manifold of complex dimension $n \geq 2$ with positive  (resp. negative) scalar curvature, and $p \in M$. Then for $\epsilon$ sufficiently small, there exists a K\"ahler metric on $Bl_pM$ with positive (resp. negative) scalar curvature in the class $\pi^*[\omega] - \epsilon^2[E]$.
\end{Theorem}
In Riemannian geometry, negative scalar curvature, and in fact negative Ricci curvature, has no obstruction beyond real dimension 2 \cite{Loh}, which is far from true in the K\"ahler case. However, we have shown that if a manifold admits a negative scalar curvature K\"ahler metric, all of its blowups do as well.

We want to emphasize the importance of the fact that the difference of the linearized scalar curvature and its formal adjoint is first order \eqref{diff}, as the appearance of further derivatives of $\phi$ would make the right hand side of the equations $(*_1)$ and $(*_2)$ not decay in $\epsilon$. This likely couldn't have been remedied by proving better bounds on the solution $\phi$ than what appear in Proposition~\ref{phiest}, as the higher order derivatives of a solution must be sufficiently big in order to approximate $S(\omega)$, since $\omega_\epsilon$ and $\tilde{\omega}_\epsilon$ are scalar flat near the exceptional divisor.

\section{Positive Scalar Curvature K\"ahler Surfaces and Further Discussion}

\subsection{Positive Scalar Curvature K\"ahler Surfaces}The most interesting consequences of the main theorem come in the case of positive scalar curvature in complex dimension 2. If $(M,\omega)$ is a K\"ahler manifold, and $\rho$ is its Ricci form, then the scalar curvature satisfies the equation
\[S(\omega)\frac{\omega^n}{n!} = \frac{1}{(n-1)!}\rho \wedge \omega^{n-1}.\]
Since $[\rho] = 2\pi c_1(M)$, integrating both sides of this equation gives
\begin{equation}\label{tsc}\int S(\omega)\frac{\omega^n}{n!} = \frac{2\pi c_1(M) \cup [\omega]^{n-1}}{(n-1)!}.
\end{equation}
The integral on the left we refer to as the \emph{total scalar curvature} of $\omega$. This shows that the total scalar curvature of a K\"ahler metric is an algebro-geometric quantity, and actually only depends on the K\"ahler class and the homotopy class of the complex structure of $M$. This formula can be used to show that the positivity of total scalar curvature has a dramatic effect on the holomorphic sections of the canonical bundle.

\begin{Proposition}[S.-T. Yau \cite{Yau}]
    Let $M$ be a compact complex manifold for which there exists a K\"ahler metric of positive total scalar curvature. Then $M$ has Kodaira dimension $-\infty$, that is, the canonical bundle of $M$ and all of its powers have no global holomorphic sections. Moreover, the converse is true in the case of compact K\"ahler surfaces.
\end{Proposition}

We will now list some basic facts about the Kodaira-Enriques classification of complex surfaces, for which a reference is \cite{Barth}. Kodaira dimension is invariant under blowup, and the Kodaira-Enriques classification classifies the K\"ahler surfaces of Kodaira dimension $-\infty$ up to blowups at points. First of all, a complex surface is called \textit{minimal} if it has  no complex submanifolds biholomorphic to $\mathbb C\mathbb P^1$ with self-intersection $-1$. Equivalently, a complex surface is minimal if it cannot be obtained from another via blowing up at a point. Furthermore, every complex surface is obtained from a minimal one via a finite sequence of blowups at points. The Kodaira-Enriques classification gives the minimal K\"ahler surfaces of Kodaira dimension $-\infty$ as follows.

\begin{Proposition}
    Any compact K\"ahler surface of Kodaira dimension $-\infty$ is obtained from $\mathbb C\mathbb P^2$ or $\mathbb P(E)$ by a finite sequence of blowups at points, where $E \rightarrow \Sigma$ is a rank two holomorphic vector bundle over a compact Riemann surface.
\end{Proposition}

A surface which is birationally equivalent to $\mathbb C\mathbb P^2$ is called \emph{rational}. Among the surfaces of Kodaira dimension $-\infty$, these are precisely those with zero first betti number \cite{Barth}. Any surface of the form $\mathbb P(E)$ as above is known as a \emph{ruled surface}. Of course, the Fubini-Study metric has positive scalar curvature, and it was shown in \cite{Yau} that ruled surfaces also admit positive scalar curvature K\"ahler metrics. Thus, our main theorem shows we can replace positive total scalar curvature in Yau's result with point-wise positive scalar curvature in the case of K\"ahler surfaces. 

Due to the gluing construction for positive scalar curvature Riemannian metrics of M. Gromov and H. B. Lawson, it was already known that these surfaces at least had Riemannian metrics of positive scalar curvature \cite{ML}. In fact, one can say even more. A complex surface is said to be of \emph{K\"ahler-type} if there exists a Riemannian metric which is K\"ahler with respect to the complex structure. This is equivalent to having even first betti number \cite{Barth}. It was shown first in the minimal case in \cite{LebSC} and then in the general case in \cite{FM} using Seiberg-Witten invariants that the existence of a positive scalar curvature Riemannian metric, not necessarily K\"ahler, on a compact complex surface of K\"ahler type, is enough to conclude the complex structure has Kodaira dimension $-\infty$. Thus, we have verified Theorem~\ref{conj}.

The contribution of the present paper is validating the presence of $(2)$ in Theorem~\ref{conj}. One can phrase this result as drawing the following equivalences. First, positive scalar curvature Riemannian geometry and positive scalar curvature K\"ahler geometry are in some sense the same on compact real $4-$manifolds for which a K\"ahler structure exists (ignoring polarizations). Second, the algebro-geometric condition of Kodaira dimension $-\infty$ for a compact complex algebraic surface can be differential geometrically restated as the existence of a K\"ahler metric with positive scalar curvature.

\subsection{Higher Dimensions and Further Questions}The natural question to ask is if Theorem ~\ref{conj} extends to higher dimensions. The equivalence with Riemannian positive scalar curvature does not persist in higher dimensions. In \cite{LebKY} the counterexample of a smooth hypersurface of degree $n + 3$ in $\mathbb C\mathbb P^{n+1}$ where $n \geq 3$ is provided. This variety is simply connected and non-spin with real dimension $\geq 6$, and thus it has a Riemannian metric of positive scalar curvature by the results of \cite{ML}. However, it has Kodaira dimension $n$ and thus no positive scalar curvature K\"ahler metric.

On the other hand, there exists many conjectures in the literature on the relationship between positive (total) scalar curvature of K\"ahler metrics (and non-K\"ahler generalizations) and algebraic geometry in arbitrary dimension; a discussion of these conjectures can be found for example in \cite{Yang}. In higher dimensions, one replaces $(5)$ in Theorem~\ref{conj} with the following notion: a complex manifold $M$ is said to be \emph{uniruled} if every point in $M$ is contained in a \emph{rational curve}, that is, a smooth complex curve biholomorphic to $\mathbb C\mathbb P^1$. We state here a K\"ahler version of the conjecture.

\begin{Conjecture}\label{Conj}
    Let $(M,\omega)$ be a compact K\"ahler manifold of arbitrary dimension. Then the following are equivalent.
    \begin{enumerate}
        \item $M$ has a K\"ahler metric for which the total scalar curvature is positive.
        \item $M$ has a K\"ahler metric of positive scalar curvature.
        \item $M$ has Kodaira dimension $-\infty$.
        \item $M$ is uniruled.
    \end{enumerate}
\end{Conjecture}
It is reasonable to contend that this may be optimistic. In dimension $n \geq 3$, K\"ahler manifolds, and in fact projective manifolds, are not closed under bimeromorphic transformation, leading to the definition of the Fujiki and Moishezon classes. The proof of our main theorem makes heavy use of this unique property for surfaces. However, \emph{balanced manifolds}, Hermitian manifolds with co-closed associated (1,1)-form, are closed under bimeromorphic transformation \cite{AB1} \cite{AB2}. Then the approach of \cite{Yau} can be mirrored to show (1) and (4) are equivalent for balanced metrics on Moishezon manifolds \cite{Balanced}, where one must consider instead the Chern scalar curvature in the case of a non-K\"ahler balanced metric. Still, this is at most a warning sign for one particular approach to Conjecture \ref{Conj}.

Recently, W. Ou \cite{OU} extended the results of \cite{BDPP} to K\"ahler manifolds, which together with the results of \cite{Yang} implies that (1) and (4) are equivalent for K\"ahler manifolds if one weakens the condition on the metric to being \textit{Gauduchon}. Namely, uniruled is equivalent to non-pseudoeffective canonical bundle (meaning it does not admit a singular metric with weakly positive curvature) \cite{OU}, which is equivalent to the existence of a Gauduchon metric with positive total Chern scalar curvature \cite{Yang}. A Gauduchon metric is a Hermitian metric whose associated (1,1)-form's $(n - 1)$th exterior power is $\partial\overline{\partial}-$closed.

The equivalence of (3) and (4) in general is already a deep conjecture in birational geometry and so we will not address it here. The proposal of the equivalence of (1) and (2) is perhaps the most optimistic. Philosophically, we interpret this as a question of whether having point-wise positive scalar curvature is truly an algebraic concept for K\"ahler metrics, or just a vestige of the differential geometric point of view, unlike the total scalar curvature which certainly is algebraic by \eqref{tsc}. Some hope for the algebraic nature of point-wise positive scalar curvature is provided by the recent work of Z. Sha \cite{SHA} on a systolic inequality for the \textit{holomorphic} 2-systole of a positive scalar curvature K\"ahler surface. In particular, the minimum volume of a holomorphic curve in a positive scalar curvature K\"ahler surface is controlled by the inverse of the minimum of its scalar curvature. 

In this vein, one can ask if point-wise positive scalar curvature, like total scalar curvature, is really a property of the K\"ahler class, that is, the \textit{polarization}. We summarize the above in the following questions.

\begin{Question}\label{Q1}
    If $(M, \omega)$ is a compact K\"ahler manifold which admits a Gauduchon metric with positive (resp. negative) total Chern scalar curvature, then must it also admit a positive (resp. negative) total scalar curvature K\"ahler metric?
\end{Question}

\begin{Question}\label{Q2}
    If $(M, \omega)$ is a compact K\"ahler manifold with positive (resp. negative) total scalar curvature, then does the K\"ahler class $[\omega]$ also admit a K\"ahler metric with point-wise positive (resp. negative) scalar curvature?
\end{Question}
The positive case of Question~\ref{Q1} is simply the missing piece of the equivalence of (1) and (4) in Conjecture~\ref{Conj} after the work of \cite{OU} and \cite{Yang}. An affirmative answer in the negative case would also provide a K\"ahler version of the relationship between the image of total scalar curvature and pseudoeffectivity of the (anti-)canonical bundle in \cite{Yang}. An affirmative answer to Question ~\ref{Q2} would show that the existence of a positive (or negative) scalar curvature K\"ahler metric is a property of the polarization.

\bibliographystyle{abbrv}

\bibliography{bibliography}

\end{document}